\documentclass[reqno]{amsart}

\usepackage{graphicx,epsfig}
\usepackage{multirow}
\usepackage{algorithm}
\usepackage{algpseudocode}
\usepackage{natbib}

\usepackage{amsmath,amssymb,latexsym, amsfonts, amscd, amsthm, xy}
\input{xy}
\xyoption{all}

\input{xy}
\xyoption{all}

\numberwithin{equation}{section}

\usepackage[usenames]{color}

\usepackage{parskip}

\makeindex
=651pt \headheight = 12pt \headsep = 0pt

\theoremstyle{plain}

\newtheorem{theorem}{Theorem}[section]

\newtheorem{corollary}[theorem]{Corollary}

\theoremstyle{definition}

\newtheorem{remark}{Remark}[section]
\newtheorem{example}[theorem]{Example}

\def\R{\mathbb{R}}

\DeclareMathOperator{\diag}{Diag}
\DeclareMathOperator*{\argmin}{argmin}

\def\bk{\boldsymbol{k}}
\def\cov{\text{Cov}}
\def\Bias{\text{Bias}}

\title{Extrinsic local regression on manifold-valued data}

\author{Lizhen Lin}
\email{lizhen.lin@austin.utexas.edu}
\address{Department of Statistics and Data Sciences, The University of Texas at Austin, Austin, TX.  }
\author{ Brian St. Thomas}
\email{brian.st.thomas@duke.edu}
\address{Department of Statistical Science, Duke University, Durham, NC}

\author{Hongtu Zhu}
\email{htzhu@email.unc.edu}
\address{UNC Gillings School of Global Public Health\\
 The University of North Carolina at Chapel Hill, Chapel Hill, NC}

\author{ David B. Dunson }
\email{dunson@duke.edu}
\address{Department of Statistical Science, Duke University, Durham, NC}

\date{}

\begin{document}
\maketitle

\begin{abstract}

We propose an extrinsic regression framework for modeling data with  manifold valued responses and Euclidean predictors. Regression with manifold responses has wide applications in shape analysis,  neuroscience, medical imaging and  many other areas. Our approach embeds the manifold where the responses lie onto a higher dimensional Euclidean space, obtains a local regression estimate in that space, and then projects this estimate back onto the image of the manifold.  Outside the regression setting both intrinsic and extrinsic approaches have been proposed for modeling i.i.d  manifold-valued data.  However, to our knowledge our work is the first to take an extrinsic approach to the regression problem.    The proposed extrinsic regression framework is general, computationally efficient and  theoretically appealing.  Asymptotic distributions and  convergence rates of the extrinsic regression estimates are derived and a large class of examples are considered indicating the wide applicability of our approach.

\textbf{Keywords}: Convergence rate;  Differentiable manifold;  Geometry; Local regression; Object data;  Shape statistics.

\end{abstract}


\section{Introduction}

Although the main focus in statistics has been on data belonging to Euclidean spaces, it is common for data to have support on non-Euclidean geometric spaces.  Perhaps the simplest example is to directional data, which lie on {\em circles or spheres}.  Directional statistics dates back to R.A. Fisher's seminal paper \citep{fisherra53} on analyzing the directions of the earth's magnetic poles, with key later developments by  \cite{watson83}, \cite{maridajpuu}, \cite{fisher87} among others.  Technological advances in science and engineering have led to the routine collection of more complex geometric data.  For example, diffusion tensor imaging (DTI) obtains local information on the directions of neural activity through $3 \times 3$ {\em positive definite matrices} at each voxel \citep{dti-ref}. In machine vision, a digital image can be represented by a set of $k$-landmarks,  the collection of which form  \emph{landmark based shape spaces} \citep{kendall84}.   In engineering and machine learning, images are often  preprocessed or reduced to a collection of \emph{subspaces}, with each data point (an image) in the sample data represented by a subspace. One may also encounter data that are stored as \emph{orthonormal frames} \citep{vecdata}, \emph{surfaces}, \emph{curves}, and \emph{networks}. 

Statistical analysis  of  data sets whose basic elements are geometric objects requires a precise mathematical characterization of the underlying space and inference is dependent on the geometry of the space. In many cases (e.g., space of positive definite matrices, spheres, shape spaces, etc), the underlying space corresponds to a \emph{manifold}.  Manifolds are general topological spaces equipped with a differentiable/smooth structure which induces a geometry that does not in general adhere to the usual Euclidean geometry.  Therefore, new statistical theory and models have to be developed for statistical inference of manifold-valued data. There have been some developments on inferences based on i.i.d (independent and identically distributed) observations on a known manifold.  Such approaches are mainly based on obtaining statistical estimators for appropriate notions of location and spread on the manifold. For example, one could base inference on the center of a distribution on the Fr\'echet mean, with the asymptotic distribution of sample estimates obtained \citep{rabi03, rabivic05,  linclt}.  There has also been some consideration of nonparametric density estimation on manifolds \citep{Bhattacharya01122010, 2013arXivStiefel, Pelletier2005297}.  \cite{rabibook} provides a recent overview of such developments.

There has also been a growing interest in modeling the relationship between a manifold-valued response $Y$ and Euclidean predictors $X$. For example, many studies are devoted to investigating how brain shape changes with age, demographic factors, IQ and other variables.  It is essential to take into  account the underlying geometry of the manifold for proper inference.  Approaches that ignore the geometry of the data can potentially  lead to highly misleading predictions and inferences.   Some geometric approaches  have been developed in the literature. For example,  \cite{fletcher11}  develops a geodesic regression model on Riemannian manifolds, which can be viewed as a counterpart of linear regression on manifolds, and  subsequent work of \cite{polyregression} generalizes polynomial regression model to the manifold.  These parametric and semi-parametric models are elegant, but may lack sufficient flexibility in certain applications.
\cite{zhuetal09} proposes a semi-parametric intrinsic regression model on manifolds, and \cite{fletcher4408977} generalizes an intrinsic kernel regression method on the Riemannian manifold,  considering applications in modeling changes in brain shape over time. \cite{yuanetal12} develops an intrinsic local polynomial model on the space of symmetric positive definite matrices, which has applications in diffusion tensor imaging. A drawback of intrinsic models is the heavy computational burden incurred by minimizing a complex objective function along geodesics, typically requiring evaluation of an expensive gradient in an iterated algorithm.  The objective functions often have multiple modes, leading to large sensitivity to start points.  Further, existence and uniqueness of the population regression function  holds  only under  relatively restrictive conditions. Therefore, usual descent algorithms used in estimation are not guaranteed to converge to a  global optima. 

 With the motivation of  developing general purpose computationally efficient, theoretically sound and practically useful regression modeling frameworks for manifold-valued response data, we propose a nonparametric extrinsic regression model by first embedding the manifold where the response resides onto some higher-dimensional Euclidean spaces.  We use equivariant embeddings, which preserve a great deal of geometry for the images.  A local regression estimate (such as a local polynomial estimate) of the regression function is obtained after embedding, which is then projected back onto the image of the manifold. Outside the regression setting,  both intrinsic and extrinsic approaches have been proposed for modeling of manifold-valued data and for mathematically studying the properties of manifolds.  However, to our knowledge,  our work is the first in taking an extrinsic approach in the regression modeling context.  Our approach is general,  has elegant asymptotic theory  and outperforms intrinsic models in terms of computation efficiency.  In addition, there is essentially no difference in inference with the examples considered.

 The  article is organized as follows. Section \ref{sec-model} introduces the extrinsic regression model.  In Section \ref{sec-simu}, we explore the full utilities of our method through applications to three examples in which  the response resides on different manifolds.  A  simulation study is carried out for data on the sphere (example \ref{ex-sphere}) applying both intrinsic and extrinsic models.  The results indicate the overall superiority of our extrinsic method in terms of computational complexity and time compared to that of  intrinsic methods.   The extrinsic models are also applied to planar shape manifolds in example \ref{ex-planar}, with an application considered to modeling the brain shape of the Corpus Callosum from an ADHD (Attention Deficit/Hyperactivity Disorder) study. In example \ref{ex-stiefel}, our method is applied to  data on the Grassmannian considering both simulated and real data. Section \ref{sec-th} is devoted to studying  the asymptotic properties of our estimators in terms of asymptotic distribution and convergence rate. 

\section{Extrinsic local regression on  manifolds}
\label{sec-model}

Let $Y\in M$ be the response  variable  in a regression model where $(M,\rho)$ is a general metric space with distance metric $\rho$. Let $X\in \R^m$ be the covariate or predictor variable. Given  data $(x_i, y_i)$ ($i=1,\ldots, m$), the goal is to model a regression relationship between $Y$ and $X$.   The typical regression framework  with  $y_i=F(x_i)+\epsilon_i$ is not appropriate here  as  expressions like $y_i-F(x_i)$ are not well-defined due to the fact that the space $M$ (e.g., a manifold) where the response variable lies  is in general not a vector space. Let $P(x,y)$ be the joint distribution of $(X, Y)$  and $P(x)$ be the marginal distribution of $X$ with marginal density $f_X(x)$. Denote $P(y|x)$ as the conditional distribution of $Y$ given $X$ with conditional density $p(y|x)$. One can define the {population regression  function or map}  $F(x)$ (if it exists) as
\begin{align}
\label{eq-model}
F(x)=\argmin_{q\in M}\int_{M}\rho^2(q,y)P(dy|x),
\end{align}
where $\rho$ is the distance metric on $M$.

Let $M$ be a $d$-dimensional differentiable or smooth manifold.  A manifold $M$ is a topological space that locally behaves like a  Euclidean space. In order to equip $M$ with a metric space structure, one can  employ a Riemannian structure, with $\rho$ taken to be the geodesic distance,  which defines an \emph{intrinsic regression function.} Alternatively, one can embed the manifold onto some higher dimensional Euclidean space via an embedding  map $J$ and  use the Euclidean distance  $\|\cdot\|$ instead. The latter model is referred to as an \emph{extrinsic regression model}. One of the potential hurdles for carrying out intrinsic analysis is that uniqueness of the population regression function in \eqref{eq-model} (with $\rho$ taken to be the geodesic distance) can be hard to verify.  \cite{Le08042014} establish several interesting and deep results for the regression framework and provide broader conditions for verifying the uniqueness of the population regression function. Intrinsic models can be  computationally expensive, since minimizing their complex objective functions typically require a gradient descent type algorithm.  In general, this requires fine tuning at each step, which results in an excessive computational burden.  Further, these  gradient descent  algorithms are  not  always guaranteed to converge  to a global minimum or only converge under very restrictive conditions.  In contrast,  the uniqueness of the population regression holds under very general conditions for extrinsic models. Extrinsic models  are extremely easy to evaluate and are  orders of magnitude faster than intrinsic models.

Let $J: M\rightarrow E^D$ be an embedding of $M$ onto some higher  dimensional ($D\geq d$) Euclidean space $E^D$ and denote the image of the embedding as $\widetilde{M}=J(M)$.  By the definition of embedding,  the differential of $J$  is a map between the tangent space of $M$ at $q$ and the tangent space of $E^D$ at $J(q)$; that is, $d_qJ: T_qM\rightarrow T_{J(q)}E^D$  is an injective map and $J$ is a homeomorphism of $M$ onto its image $\widetilde{M}$. Here $T_qM$ is the tangent space of $M$ at $q$ and $ T_{J(q)}E^D$ is the tangent space of $E^D$ at $J(q)$.   Let $||\cdot||$ be the Euclidean norm.
In an extrinsic model, the true extrinsic regression function is defined as
\begin{align}
F(x)&\nonumber=\argmin_{q\in M}\int_{M}||J(q)-J(y)||^2 P(dy|x)\\
    &=   \argmin_{q\in M}\int_{\widetilde M}||J(q)-z||^2  \widetilde{P}(dz|x)
\end{align}
where $\widetilde{P}(\cdot\mid x)=P(\cdot\mid x)\circ J^{-1}$ is the conditional probability measure on $J(M)$ given $x$ induced by the conditional probability measure  $P(\cdot\mid x)$ via  the embedding $J$.

We now proceed to propose an estimator for $F(x)$. Let $K:\mathbb R^m \rightarrow \mathbb R$ be a multivariate  kernel function such that $\int_{\mathbb R^m} K(x)dx=1$ and   $\int_{\mathbb R^m}x K(x)dx=0$. One can take $K$ to be a product of $m$ one-dimensional kernel functions for example. Let $H=\diag(h_1,\ldots, h_m)$  with $h_i>0$ ($i=1,\ldots, m$)  be the bandwidth vector and  $|H|=h_1\ldots h_m$. Let $K_H(x)=\frac{1}{|H|}K(H^{-1}x)$ and
\begin{equation}
\label{eq-kesti}
\widehat{F}(x)=\argmin_{y\in E^D}\sum_{i=1}^n \dfrac{K_H(x_i-x)||y-J(y_i)||^2}{\sum_{i=1}^nK_H(x_i-x)}=\sum_{i=1}^n\dfrac{J(y_i)K_H(x_i-x)}{\sum_{i=1}^nK_H(x_i-x)},
\end{equation}
which is basically a weighted average of points $J(y_1),\ldots, J(y_n)$.
We are now ready to define the \emph{extrinsic kernel estimate of the regression function} $F(x)$ as
\begin{equation}
\label{eq-extrinsic}
\widehat{F}_E(x)=J^{-1}\left(\mathcal P(\widehat{F}(x))\right)=J^{-1}\left(\argmin_{q\in \widetilde{M}}||q-\widehat{F}(x)||\right),
\end{equation}
where $\mathcal P$ denotes the projection map onto the image $\widetilde{M}$.
Basically, our estimation procedure consists of two steps. In step one, it calculates a local regression estimate on the Euclidean space after embedding. In step two,  the estimate obtained in step one is projected back onto the image of the manifold.

\begin{remark}
The embedding $J$ used in the extrinsic regression model is in general not unique.  It is desirable to have  an embedding that preserves as much geometry as possible.  An \emph{equivariant embedding}   preserves a substantial amount of geometry.   Let $G$ be some large Lie group acting on $M$. We say that $J$ is an equivariant embedding if we can find a group homomorphism  $\phi: G\rightarrow GL(D, \mathbb R)$ from $G$ to the general linear group $GL(D, \mathbb R)$ of degree $D$  such that
\begin{align*}
J(gq)=\phi(g)J(q)
\end{align*}
for any $g\in G$ and $q\in M$. The intuition behind equivariant embedding  is that the image of  $M$ under  the group action of the Lie group $G$  is preserved by the group action of $\phi(G)$ on the image, thus preserving many  geometric features. Note that the choice of embedding is not unique and in some cases constructing an equivariant embedding can be a non-trivial task, but in most of the cases a natural embedding would arise and such embeddings can often be verified as equivariant.
\end{remark}

\begin{remark}
Alternatively,  we can obtain  some robust estimator under our proposed framework. The regression estimate is taken as the projection of the following estimator onto the image $\widetilde{M}$ of $M$ after an embedding $J$. We can call it the \emph{extrinsic median regression model}. Specifically, we define

\begin{align}
\label{eq-kestirobust}
\widehat{F}(x)=\argmin_{y\in E^D}\sum_{i=1}^n \dfrac{K_H(x_i-x)||y-J(y_i)||}{\sum_{i=1}^nK_H(x_i-x)}\; \;\text{and}\; \;\widehat{F}_E(x)=J^{-1}\left(\argmin_{q\in \widetilde{M}}||q-\widehat{F}(x)||\right).
\end{align}
One can use the Weizfield formula \citep{wes37} in calculating the weighted median of \eqref{eq-kestirobust} (if it exists).  Such estimates can be shown to be robust to outliers and contaminations.
\end{remark}

\begin{remark}
\label{re-poly}
A kernel estimate is obtained first in \eqref{eq-kesti} before projection. However,  the framework can be easily generalized using higher order local polynomial regression estimates (of degree $p$)\citep{fan1996local}. For example, one can have a \emph{local linear estimator}  \citep{fan1993} for $\widehat{F}(x)$ before projection. That is,  for any $x$, let
\begin{align}
\label{eq-localinear}
(\hat{\boldsymbol\beta}_0,\hat{\boldsymbol\beta}_1)&=\argmin_{\boldsymbol\beta_0,\boldsymbol\beta_1 }\sum_{i=1}^n \left\|J(y_i)-\boldsymbol\beta_0-\boldsymbol\beta_1^t(x_i-x)\right\|^2K_H(x_i-x).
\end{align}
Then, we have
\begin{align}
\widehat{F}(x)&=\hat{\boldsymbol\beta}_0(x),\\
\widehat{F}_E(x)&=J^{-1}\left(\mathcal P(\widehat{F}(x))\right)=J^{-1}\left(\argmin_{q\in \widetilde{M}}||q-\widehat{F}(x)||\right) \label{eq-polyesti}.
\end{align}
The  properties of the estimator $\widehat{F}_E(x)$ where  $\widehat{F}(x)$ is given by the general $p$th local polynomial estimator of $J(y_1),\ldots, J(y_n)$ are explored in Theorem \ref{th-poly}.
\end{remark}

Note that our work addresses different problems  from that of \cite{mingyenjasa}, which provides an elegant framework for high dimensional data analysis and manifold learning  by first performing  local linear regression  on a tangent plane estimate of a lower-dimensional manifold  where   the high-dimensional data concentrate.

\section{Examples and applications}
\label{sec-simu}
The proposed extrinsic regression framework is  very general and has appealing asymptotic properties as  will be shown in Section 4.   To illustrate the wide applicability of our approach and validate its finite sample performance, we carry out  a study by  applying our method to various  examples with the response taking values in many well-known manifolds.  For each of the examples considered, we provide details on the embeddings, verify such embeddings are equivariant, and give explicit expressions for the projections to obtain the final estimate in each case.
In example \ref{ex-sphere}, we simulate data from a 2-dimensional sphere  and compare the estimates from our extrinsic regression model with that of an intrinsic model. The result indicates that  the extrinsic models clearly outperform  the intrinsic models by orders of magnitude  in terms of computational complexity and time.
In example \ref{ex-planar}, we study a data example with response from a planar shape, in which  the brain shape of the subjects are represented by landmarks on the boundary. Example \ref{ex-stiefel} provides  details  of the estimator when the responses take values on a Stiefel or Grassmann manifold. The method is illustrated with a synthetic data set and small financial time series data set, both of which have subspace responses of possibly mixed dimension and covariates,  which are the corresponding time points.

\begin{example}
\label{ex-sphere}
Statistical analysis on i.i.d data from  the 2-dimensional sphere $S^2$,  often called  \emph{directional statistics},  has a long history \citep{fisherra53, watson83, maridajpuu, fisher87}. Recently,   \cite{wang_lerman2015} applied a nonparametric Bayesian approach to an example with response on the circle $S^1$.  In this example, we  work out the details in an extrinsic regression model  with the responses lying on a $d$-dimensional sphere $S^d$.  The model is illustrated with data $\{(x_i, y_i)$, $i=1,\ldots, n$\},  where  $y_i\in S^2$.

Note that  $S^d$ is a submanifold of $\R^{d+1}$;  therefore, the  inclusion map $\imath $ serves as a natural embedding onto $\mathbb R^{d+1}$.  It is easy to check that the embedding is equivariant  with the Lie group $G=SO(d + 1)$, the special orthogonal group of $(d+1)$ by $(d+1)$ matrices $A$ with $AA^T=1$ and $|A|=1$. Take the homomorphism map from $G$ to $GL(d+1, \mathbb R)$ to be the identity map. Then it is easy to see that $J(gp)=gp=\phi(g)J(p)$, where $g\in G$ and $p\in S^d$.

Given $J(y_1),\ldots, J(y_n)$, one first obtains  $\widehat{F}(x)$ as given in \eqref{eq-kesti}. Its projection onto the image $\widetilde{M}$ is given by
\begin{equation}
\widehat{F}_E(x)=\widehat{F}(x)/||\widehat{F}(x)||,\;\text{when}\; \widehat{F}(x)\neq 0.
\end{equation}

There are many well defined parametric distributions on the sphere.   A common and useful distribution is the von Mises-Fisher distribution  \citep{fisherra53} on the unit sphere, which has the following density with respect to the normalized volume measure on the sphere:
$$p_{MF} (y;\mu,\kappa) \propto \exp(\kappa \mu^T y),$$
where $\kappa$ is a concentration parameter with $\mu$  a location parameter and  $E(y) = \mu$ holds.  We simulate the data from  the unit sphere by letting  the mean function be covariate-dependent. That is, let
$$\mu =  \frac{\beta \circ x}{\lvert \beta \circ x \rvert},$$ where $\beta \circ x$ is the Hadamard product $(\beta_1x^1,\ldots, \beta_mx^m)$.

For this example,  we will use data generated by the following model
\begin{align}
\label{eq-data}
\beta \sim& N_3(0,I),\; x_i^1 \sim N(0,1),\; x_i^2 \sim N(0,1),\; x_i^3 = x_i^1*x_i^2,\\ \nonumber
y_i \sim& MF\left(\mu_i, \kappa\right), \; \mu_i = \frac{\beta \circ x_i}{\lvert \beta \circ x_i \rvert},\;i = 1, \ldots, n,\\ \nonumber
\kappa\; & \text{ some fixed known value}.
\end{align}

As an example of what the data looks like, we generate one thousand $(n = 1000)$ observations from the above model with $\kappa = 10$ so that realizations are near their expected value. Figure \ref{fig:sphere2}  shows this example in which 100 predictions from the extrinsic model are plotted against their true values using 900 training points.  To select the bandwidth $h$ we use 10-fold cross-validation with $h$ ranging  from $[.1,.2, \ldots, 1.9, 2]$ and choose the value that gives minimum average mean square error.  Residuals for the mean square error are measured using the intrinsic distance, or great circle distance, on the sphere.

\begin{figure}[ht]
\begin{center}
\includegraphics[scale=.6]{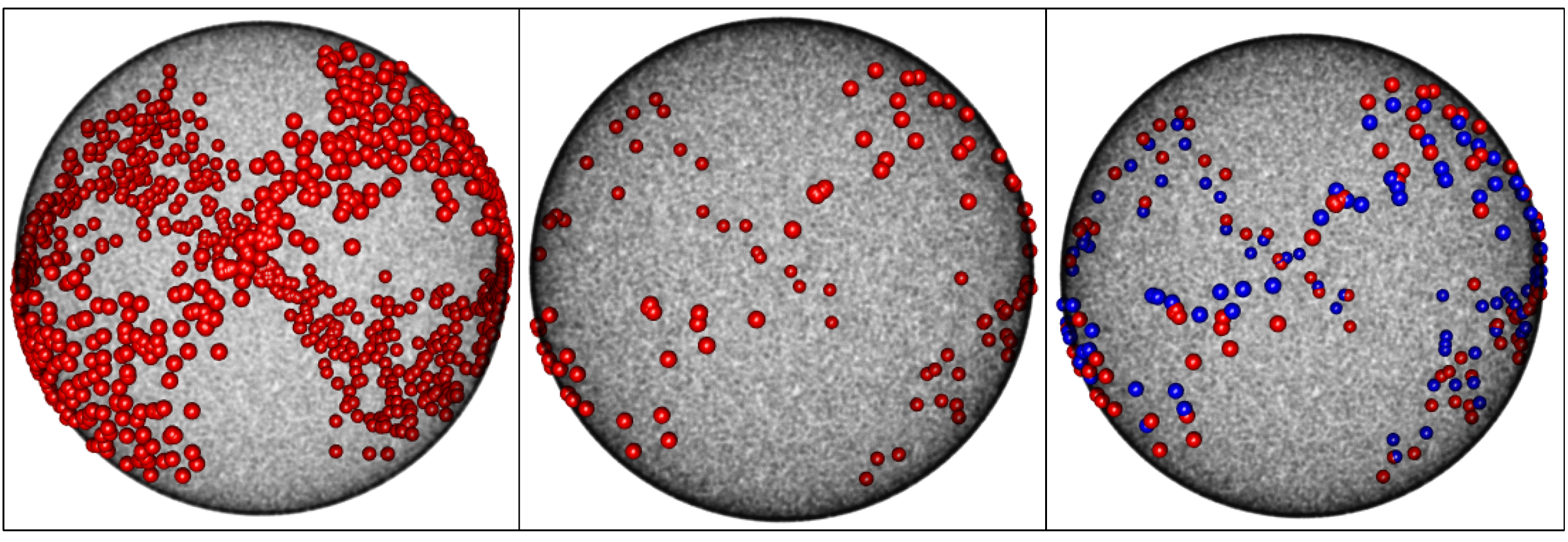}
\end{center}
\vspace{-.1in}
 \caption{\textbf{Left} The training values on the sphere. \textbf{Middle} The held out values to be predicted through extrinsic regression. \textbf{Right} The extrinsic predictions (blue) plotted against the true values (red).}
 \label{fig:sphere2}
 \end{figure}

To illustrate the utility and advantages of extrinsic  regression models, we compare our method to an intrinsic kernel regression model that uses intrinsic distance of the sphere to minimize the objective function. Computations on the sphere are  in general not  as  intensive compared to more complicated manifolds such as shape spaces,  etc, but it still requires an iterative algorithm, such as gradient descent, for  the intrinsic model in order to obtain a kernel regression estimate. The following simulation results  demonstrate extrinsic kernel regression gives at least as accurate estimates as intrinsic kernel regression but in much less computation time even for $S^2$.

\textbf{Comparison with an intrinsic kernel regression model:}  The intrinsic kernel regression estimate minimizes the objective function $f(y) = \sum_{i = 1}^n w_i d^2(y,y_i)$, where $y$ and $y_i$ are points on the sphere $S^2$, $w_i$ are determined by the Gaussian kernel function, and $d(\cdot,\cdot)$ in this case is the greater circle distance. Then the gradient of $f$ on the sphere is given by
\begin{align*}
\nabla f(y) = \sum_{i = 1}^n w_i2d(y,y_i)\frac{\log_y(y_i)}{d(y,y_i)}= \sum_{i = 1} 2 w_i \frac{\operatorname{arccos}(y^Ty_i)}{\sqrt{1 - (y^Ty_i)^2}}(y_i - (y^Ty_i)y),
\end{align*}
where $\log_y(y_i)$ is the $\log$ map or the inverse exponential map on the sphere.
Estimates for $y$ can be obtained through a gradient descent algorithm with step size $\delta$ and error threshold $\epsilon$. We applied the intrinsic and extrinsic models to the same set of data using the Gaussian kernel function. 

Twenty different data sets of 2000 observations were generated from the above sphere regression model with von-Mises Fisher concentration parameter $\kappa = \{1, 2, \ldots, 20\}$. Of the 2000 observations, 50 were used to check the accuracy of the extrinsic and intrinsic estimates. To see the effect of training sample size on the quality of the estimates, the estimates were also made on subsets of the 1950 training observations, starting with 2 observations and increasing to all 1950 observations. The same training observations were always used for both models. In both models, the bandwidth was chosen through 10-fold cross validation. The intrinsic kernel regression was fit with step size $\delta = .01$ and error threshold $\epsilon = .001$. The performance of the two models are compared in terms of MSE and predictive MSE. The MSE is calculated using the greater circle distance between predicted values and the true expected value, while predictive MSE is calculated using the greater circle distance between the predicted values and the realized values. The performance results using 50 hold out observations can be seen in Figure ~\ref{fig:mse_comp}. 

\begin{figure}[h]
\begin{center}
\includegraphics[scale=.5]{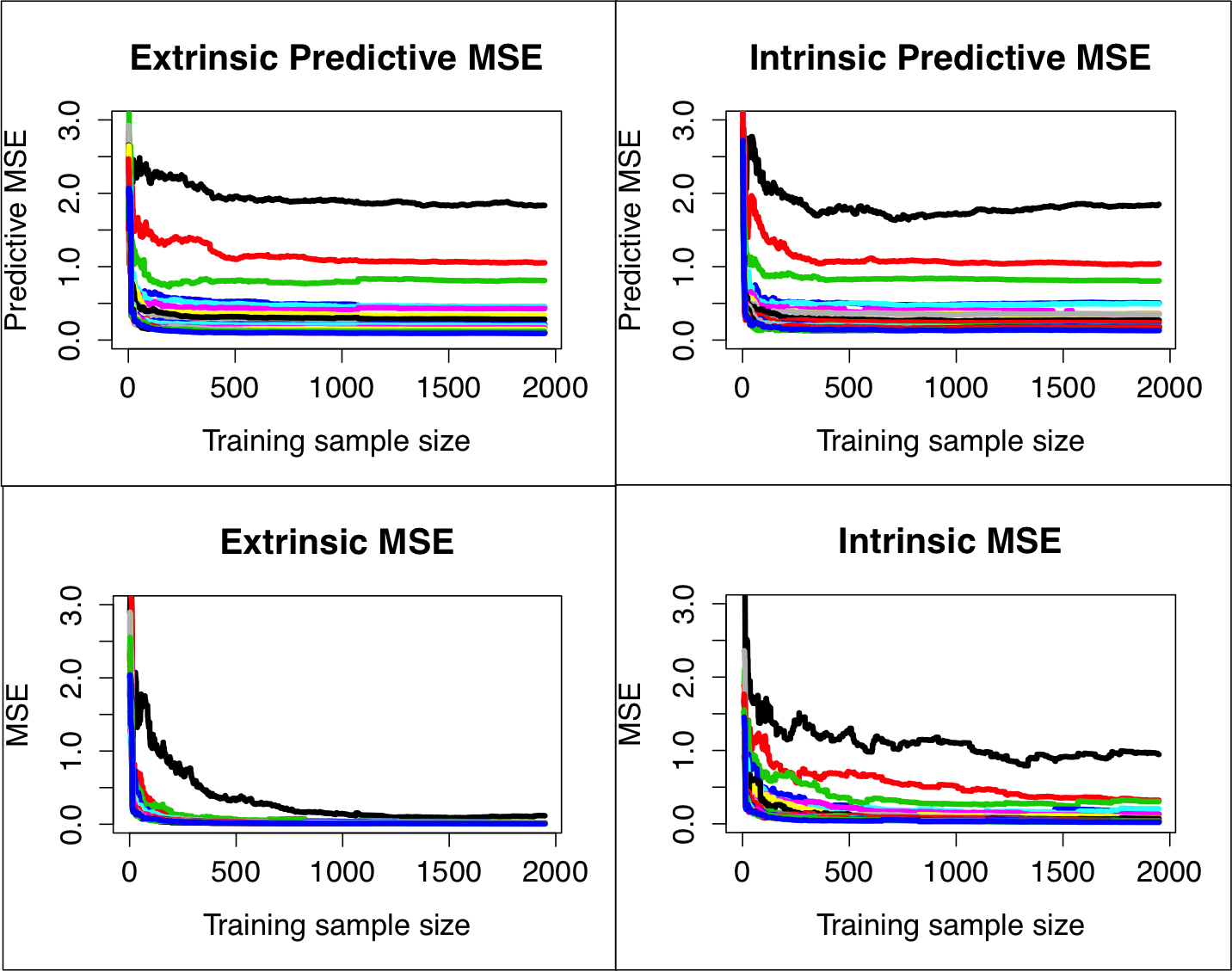}
\caption{The performance of extrinsic and intrinsic regression models on 50 test observations from sphere regression models with concentration parameters from 1 to 20. Each color corresponds to a concentration parameter. The extrinsic and intrinsic models have similar performance in predictive MSE with low concentration parameters. However in terms of MSE, the extrinsic model appears to perform better with lower sample sizes even with lower concentration parameters.}
\label{fig:mse_comp}
\end{center}
\end{figure}
\end{example}

Predictive MSE does not converge to 0 because the generating distribution has a high variance; however, as the concentration increases,  the predictive MSE does approach 0. The extrinsic and intrinsic kernel regressions perform similarly with large sample sizes. The extrinsic kernel regression drops in predictive MSE faster than the intrinsic model, which may stem from only having the kernel bandwidth as a tuning parameter which can be selected more easily than choosing the bandwidth, step-size, and error thresholds even through cross-validation.

A significant advantage of the extrinsic kernel regression is the speed of computation. Both methods were implemented in C++ using Rcpp \citep{refBrianadd}, and resulted in up to a 60$\times$ improvement in speed in making a single prediction using all of the training observations. For speed comparisons, a single prediction was made given the same number of test observations, and the time to produce the estimate was recorded. Each of these trials was done five times, and we compare the mean time to producing the estimate in Figure ~\ref{fig:speed}.

\begin{figure}[ht]
\begin{center}
\includegraphics[scale=.55]{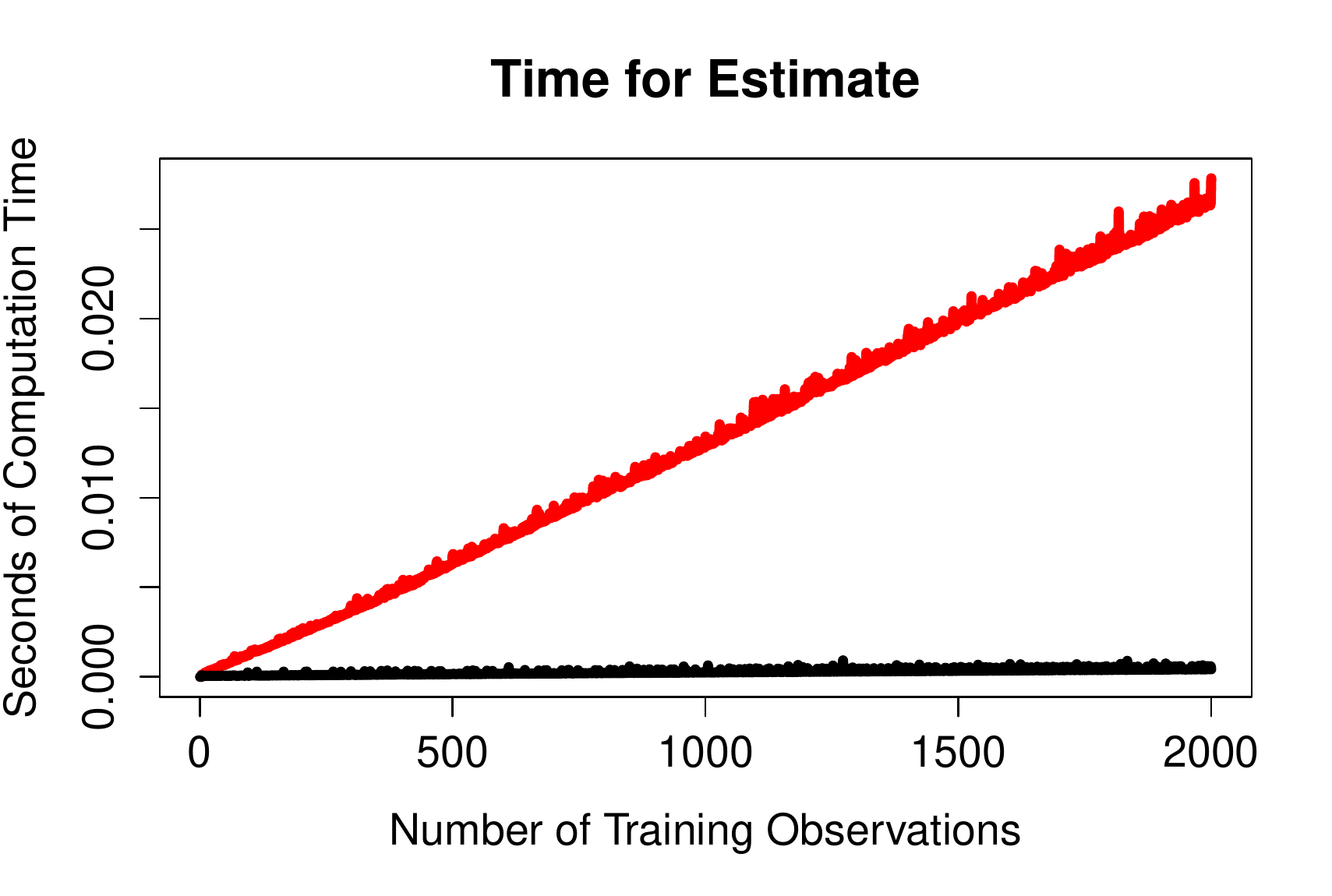}
\includegraphics[scale=.55]{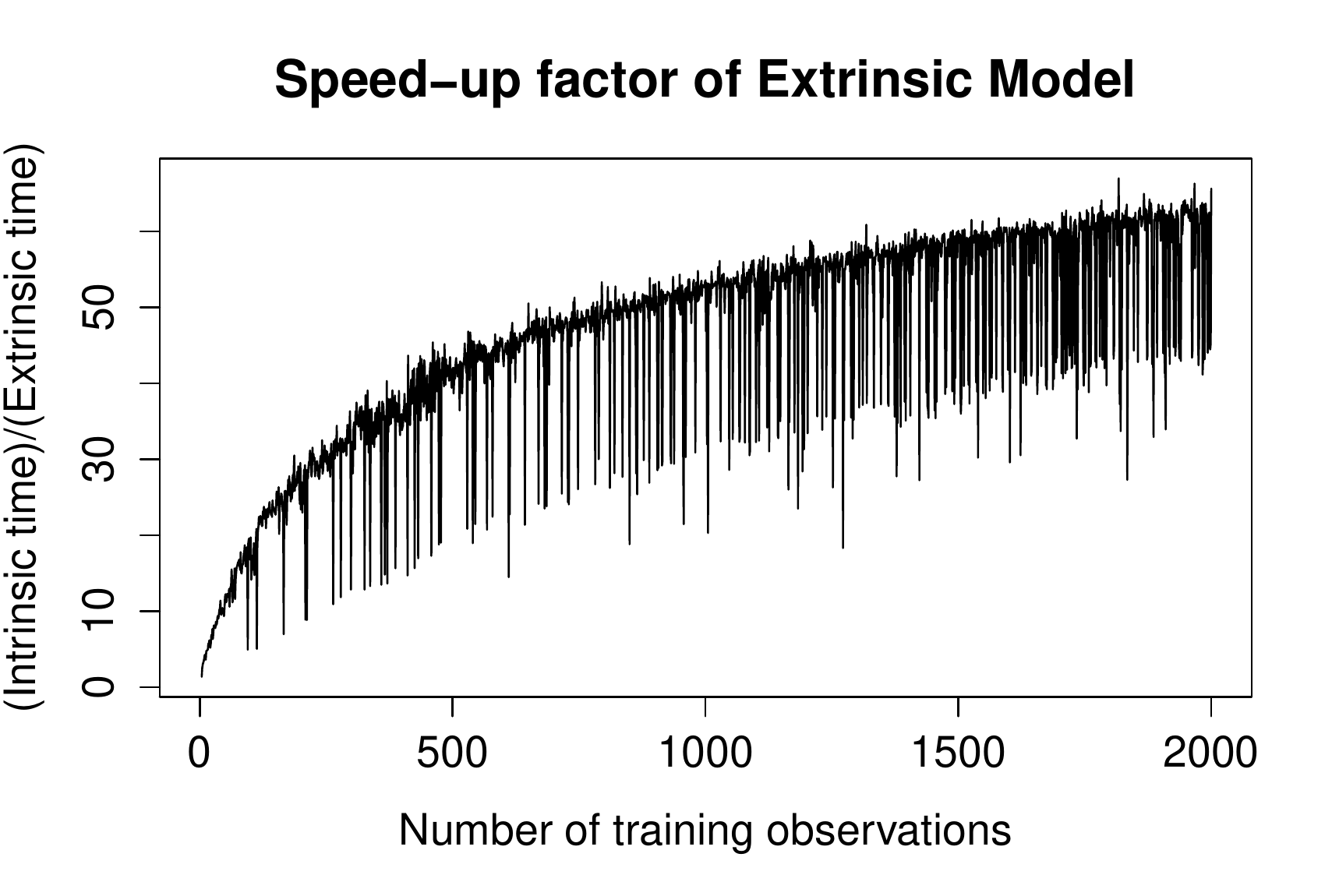}
\end{center}
\caption{Speed comparisons between the extrinsic and intrinsic kernel regressions as a function of the number of training observations. The average seconds to produce an estimate for a single test observation are plotted in red for the intrinsic model, and black for the extrinsic model. The multiple between the speed for the intrinsic and extrinsic estimates plotted are also plotted for reference.}
\label{fig:speed}
\end{figure}

Note that the same kernel weights are computed in both algorithms, so the difference is attributable to the gradient descent versus extrinsic optimization procedures. Since the speed comparisons were done for computing a single prediction and the difference is due almost entirely to the gradient descent steps, making multiple predictions results in an even more favorable comparison for the extrinsic model. This experiment shows that the extrinsic kernel regression applied to sphere data performs at least as well on prediction and can be computed significantly faster.

\begin{example}
\label{ex-planar}
We now consider an example with planar shape responses. Planar shapes are one of the most important classes of landmark based shapes spaces. Such spaces were defined by \cite{kendall77} and \cite{kendall84} with pioneering work by \cite{books1} motivated from applications on biological shapes. We now describe the geometry of the space  which will be used in  obtaining regression estimates for our model.  Let $z=(z_1,\ldots, z_k)$ with $z_1,\ldots, z_k\in \R^2$ be a set of $k$ landmarks.  Let $<z>=(\bar{z},\ldots, \bar{z})$ where $\bar{z}=\sum_{i=1}^k  z_i/k.$ Denote  $u=\dfrac{z-<z>}{||z-<z>||}$ which  can be viewed as an element on  the sphere $S^{2k-3}$,  which is called the pre-shape.
   The planar shape $\Sigma_2^k$ can now be represented as the quotient of the pre-shape under the group action by $SO(2)$, the 2 by 2 special orthogonal group. That is,
   $\Sigma_2^k=S^{2k-2-1}/SO(2)$.  $\Sigma_2^k$ can be shown to be  equivalent to the complex projective space $\mathbb C\mathbb P^{k-2}$.
Therefore, a point on the planar shape can be identified  as the orbit or equivalent of $z$  which we denote by $\sigma(z)$.
   Viewing $z$ as elements in the complex plane,  one can embed $\Sigma_2^k$ onto the $S(k,\mathbb C)$, the space of $k\times k$ complex Hermitian matrices via the Veronese-Whitney embedding (see \cite{rabivic05}, \cite{rabibook}):
\begin{equation}
\label{eq-planaremb}
J(\sigma(z))=uu^*=((u_i\bar{u}_j))_{1\leq, i,j\leq k}.
\end{equation}
One can verify the Veronese-Whitney  embedding is equivariant (see \cite{kendall84}) by taking the Lie group $G$ to be  special unitary group $SU(k)$ with
$$SU(k)=\{A\in GL(k, \mathbb C), AA^*=I, det(A)=I\}.$$
The action is on the left,
\begin{align*}
A\sigma(z)=\sigma(Az).
\end{align*}
The  homomorphism map $\phi$ is taken to be
$$\phi: S(k,\mathbb C)\rightarrow S(k, \mathbb C): \phi(A)\widetilde A=A\widetilde A A^*.$$
Therefore, one has
\begin{align*}
J(A\sigma(z))=Auu^*A^*=\phi(A)J(\sigma(z)).
\end{align*}

We now describe the projection after $\widehat{F}(x)$ is given  by \eqref{eq-kesti},  where $J(y_i)$ ($i=1,\ldots, n$) are obtained using the equivariant embedding given in \eqref{eq-planaremb}. Letting $v^T$ be the eigenvector corresponding to largest eigenvalue of $\widehat{F}(x)$, by a careful calculation, one can show that  the projection of  $\widehat {F}(x)$ is given by $$\mathcal P_{J(M)}\left(\widehat{F}(x)\right)=v^T\bar{v}.$$
 Therefore, the extrinsic kernel regression estimate is given by
\begin{equation}
\widehat{F}_E(x)=J^{-1}(v^T\bar{v}).
\end{equation}

\textbf{Corpus Callosum (CC) data set:} We study  ADHD-200 dataset \footnote{http://fcon\_1000.projects.nitrc.org/indi/adhd200/}  in which the shape contour of the brain
Corpus Callosum  are recorded for each subject along with variables such as gender, age, and ADHD diagnosis.   The subjects consist of patients who  are diagnosed with  ADHD.   50 landmarks were placed  outlining the CC shape for 647 patients  for the ADHD-200 dataset. The age of the patients range from 7 to 21 years old, with 404 typically developing children and 243 individuals diagnosed with some form of ADHD. The original data set differentiates between types of ADHD diagnoses, and we simplify the problem of choosing a kernel by using a binary response for an ADHD diagnosis.

According to the findings in \cite{hongtu15}, there is not a significant effect of gender on the area of different segments of the CC; however diagnosis and the interaction between diagnosis and age were found to be statistically significant ($ p <.01$). With knowledge of these results, we performed the extrinsic kernel regression method for the CC planar shape response using diagnosis, $x^1$, and age, $x^2$, for covariates. The choice of kernel between two sets of covariates $x_1 = (x^1_1, x^2_1)$ and $x_2 = (x^1_2, x^2_2)$ is
$$K_H(x_1,x_2) = \begin{cases} \exp\left(-\frac{(x_1^2 - x_2^2)^2}{h}\right)/h^2 &\mbox{if } x_1^1 \equiv x_2^1  \\
 0 & \mbox{if } x_1^1 \not \equiv x_2^1 .\\
 \end{cases}$$
Although \cite{hongtu15} explores  clustering the shape by specific diagnosis, we visualize how the CC shape develops over time by making predictions at different time points. We show predictions for ages 9, 12, 16, and 19 year old children of ADHD diagnosis or typical development. The results can be seen in Figure ~\ref{fig:cc8_12}.

\begin{figure}[ht]
\begin{center}
\includegraphics[scale=.55]{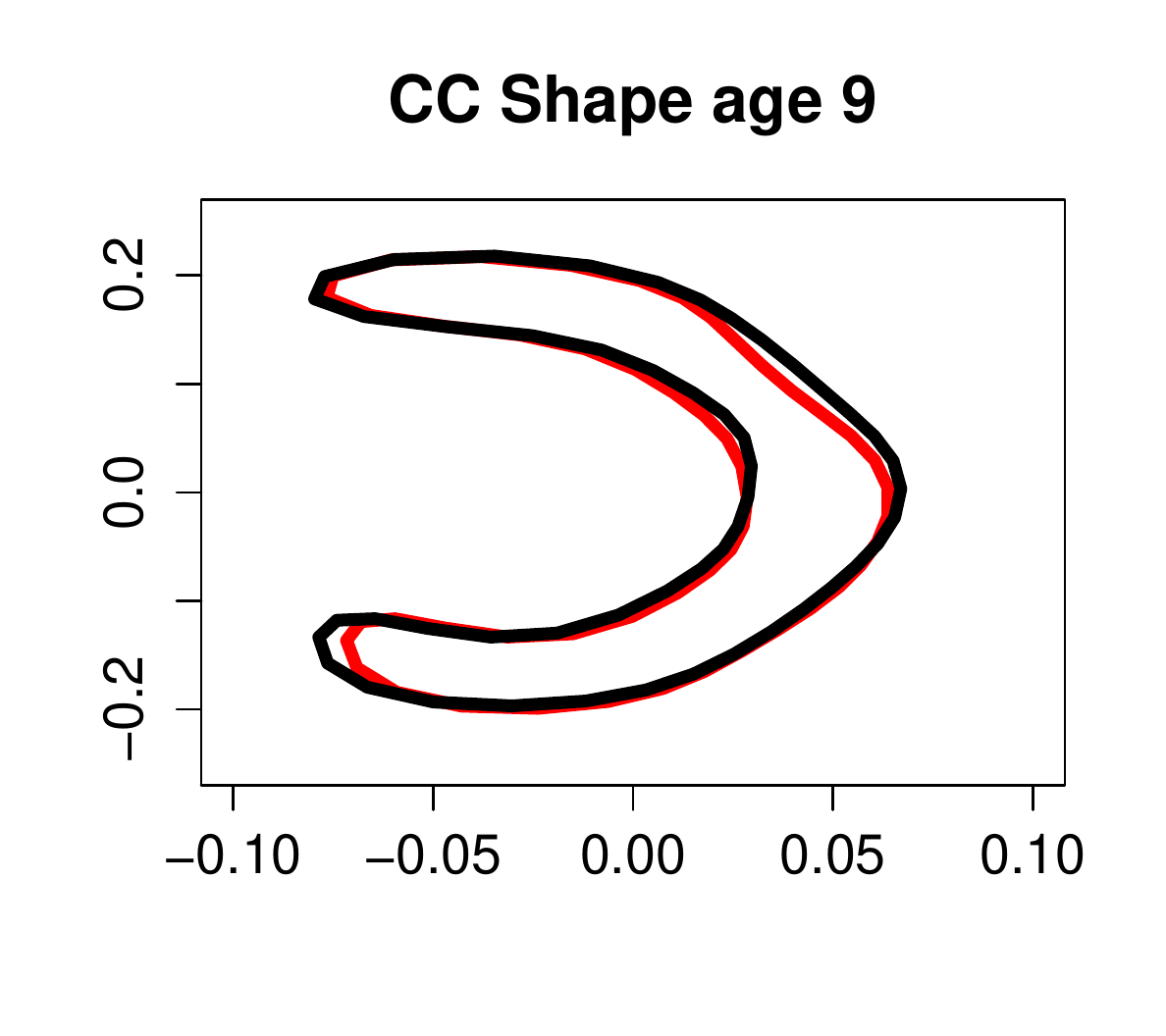}
\includegraphics[scale=.55]{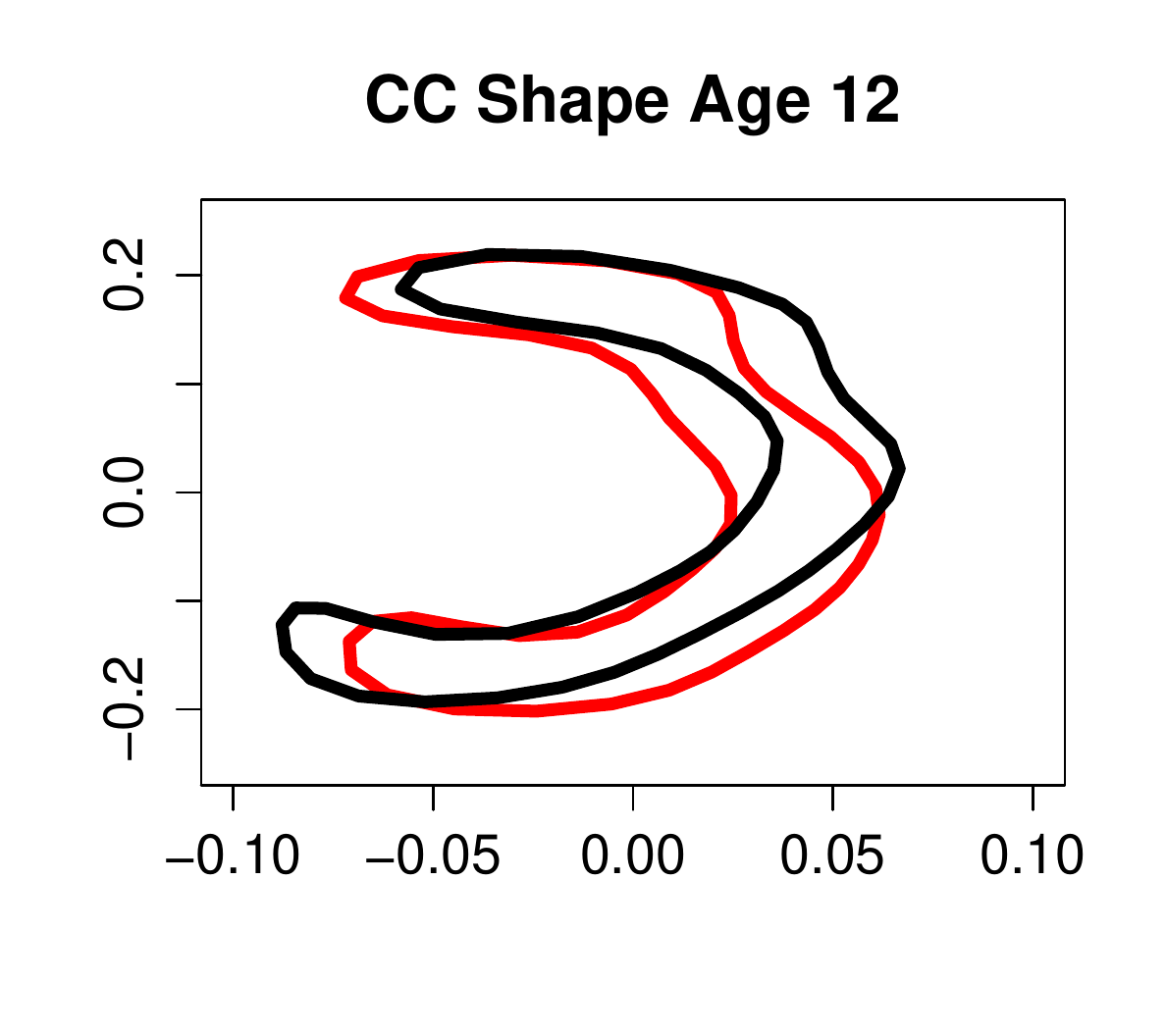}
\includegraphics[scale=.55]{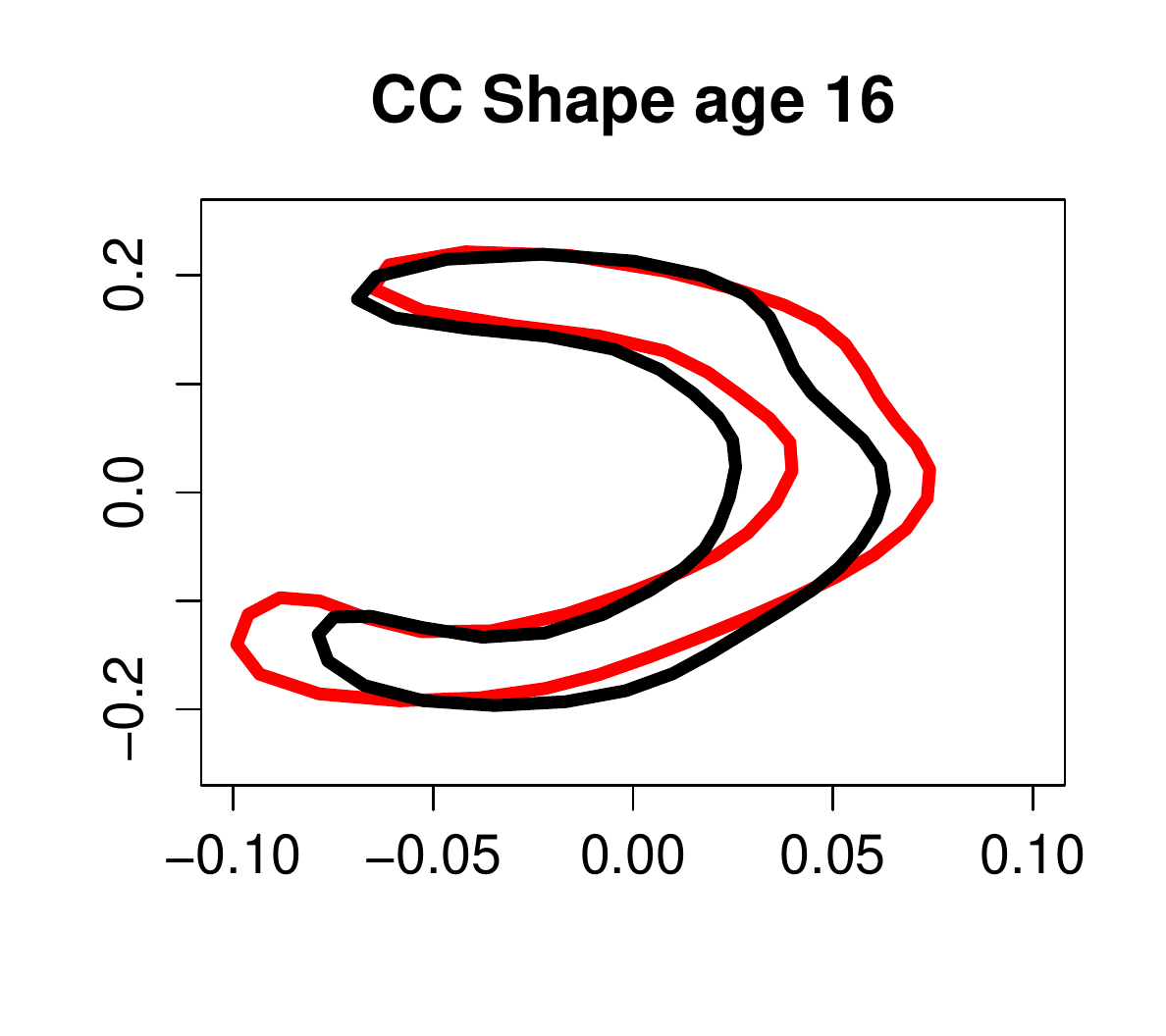}
\includegraphics[scale=.55]{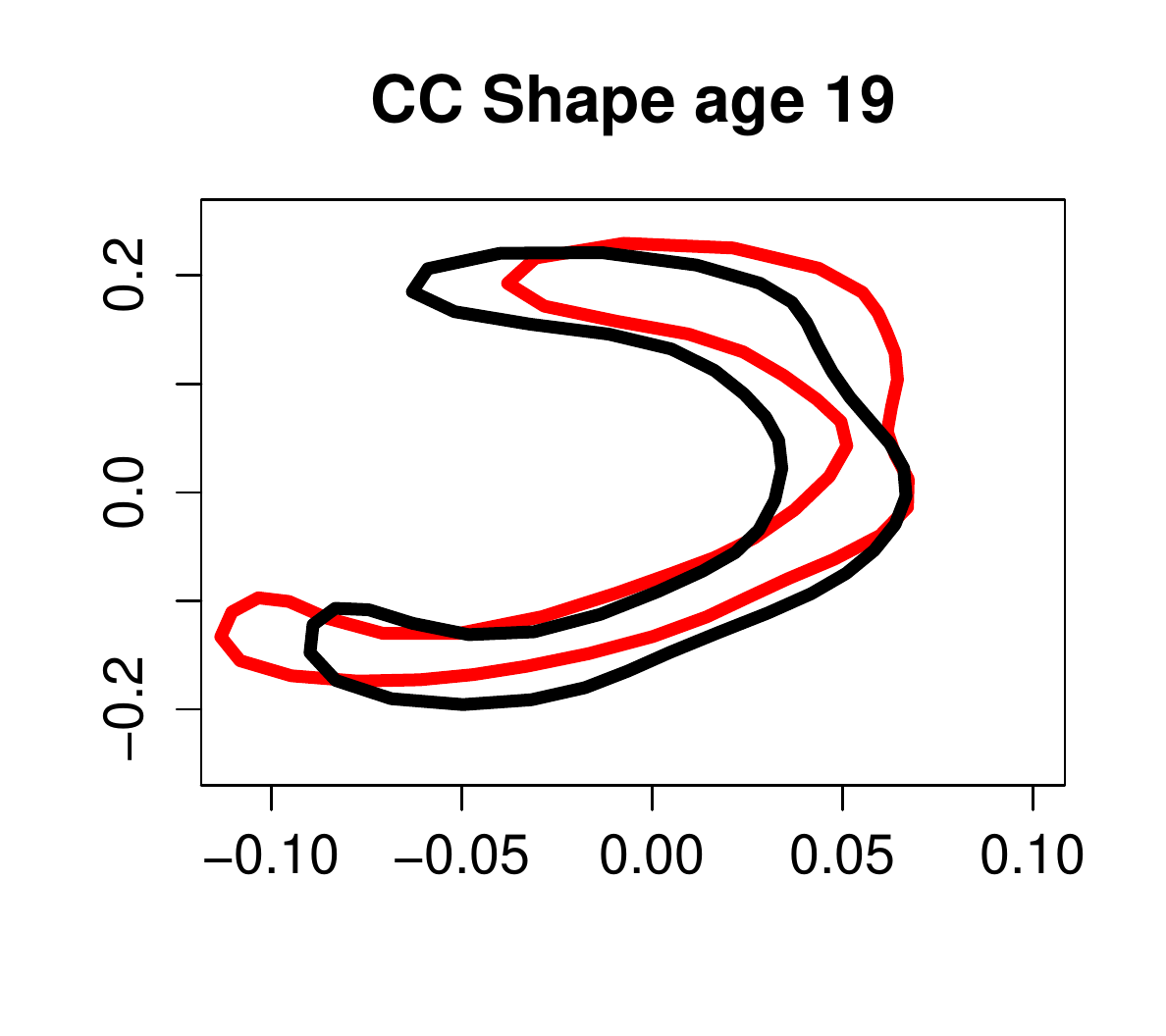}
\end{center}
\caption{Predicted CC shape for children ages 9, 12, 16, and 19. The black shape corresponds to typically developing children, while the red shape corresponds to children diagnosed with ADHD. Kernel regression allows us to visualize how CC shape changes through development. Here sections of CC appear smaller in ADHD diagnoses than in normal development.}
\label{fig:cc8_12}
\end{figure}

What we can observe from the two plots is that the CC shapes for the 8 year olds seem to be close, but by age 12 the shapes have diverged substantially, with shrinking of the CC being apparent in later years in development. This quality of the CC shapes between ADHD and normal development is consistent with results found in the literature \citep{hongtu15}.

In previous studies, ADHD diagnoses were clustered using the shape information to predict the diagnosis class, and the centroid of the cluster is the predicted shape for that class \citep{hongtu15}. Our method adds to this analysis by taking the diagnosis and predicting the CC shape as a function of age. Our method also has the benefit of evaluating quickly, making selection of the bandwidth for the kernel through cross-validation feasible.
\end{example}

\begin{example}
\label{ex-stiefel}
We now consider another two classes of important  manifolds,  Stiefel manifolds and  Grassman manifolds (Grassmannians).
The Stiefel manifold, $V_k(\mathbb R^m)$, is the collection of $k$ orthonormal frames in $\mathbb R^m$. That is,   the Stiefel manifold consists of  the set of ordered $k$-tuples of orthonormal vectors in $\R^m$, which can be represented as $\{X\in S(m,k), XX^T=I_m  \}$. The Stiefel manifold includes the $m$ dimensional sphere $S^m$ as a special case with $k$=1 and $O(m)$ the  orthogonal group when $k=m$.   Examples of data on the Stiefel manifold include  the orbit of the comets and the vector cardiogram.  Applications of Stiefel manifold are present in  earth sciences, medicine, astronomy, meteorology and biology.  The Stiefel  manifold is a compact manifold of dimension $km-k-k(k-1)/2$  and it is a submanifold of $\R^{km}$.  The inclusion map can be further shown to be an equivariant embedding with the Lie group taken to  the orthogonal group $O(m)$.

Given $\widehat{F}(x)$ obtained by kernel regression after embedding the points $y_1,\ldots, y_n$ on the Stiefel manifold to the Euclidean space $\R^{km}$, the next step is to obtain the projection of  $\widehat{F}(x)$  onto $\widetilde{M}=J(M).$  We first make an \emph{orthogonal decomposition} of  $\widehat{F}(x)$  by letting $\widehat{F}(x)=US$, where $U\in V_{k,m}$, which can be viewed as the orientation of  $\widehat{F}(x)$   and $S$ is positive semi-definite, which has the same rank as  $\widehat{F}(x)$.   Then the projection of  $\widehat{F}(x)$ (or projection set)  is given by
\begin{align*}
\mathcal P_{\tilde{M}}( \widehat{F}(x) )=\{U\in V_{k,m}:  \widehat{F}(x) =U( \widehat{F}(x)^T \widehat{F}(x) )^{1/2}  \}.
\end{align*}
See Theorem 10.2 in \cite{rabibook} for a proof of the results.  Then the projection is unique, that is, the above set is a singleton if and only  if $\widehat{F}(x)$  is of full rank.


The Grassmann manifold or the Grassmannian $Gr_k(\R^m)$ is  the space of all the subspaces of a fixed dimension $k$ whose basis elements are vectors in $\R^m$, which is closely related to the Stiefel manifold $V_{k,m}$. Recall a subspace can be viewed as the span of an orthonormal basis. Let $\boldsymbol{v}=\{v_1,\ldots, v_k\}$ be such an orthonormal basis for a subspace on the Grassmannian. Note that the order of the vector does not matter unlike in the case of Stiefel manifold. For any two elements on the Stiefel manifold whose span corresponds to the same subspace, there exists an orthogonal transformation (mapped by a orthogonal matrix in $O(k)$) between the two orthonormal frames.  These two points will be identified as the same point on the Grassman manifold. Therefore, the Grassmannian can be viewed as the collection of the equivalent classes on the Stiefel manifold, i.e., a quotient space under the group action of $O(k)$, the $k$ by $k$ orthogonal group. Then one has $Gr_k(\R^m)=V_k(\mathbb R^m)/O(k)$.  There are many applications of  Grassmann manifolds, in which the subspaces are the basic element in signal processing, machine learning and so on.

The equivariant embedding for $Gr_k(\R^m)$ also exists \citep{chikuse}. 	Let $X\in V_{k,m}$ be a representative element of the equivalent classes in  $Gr_k(\R^m)=V_k(\mathbb R^m)/O(k)$.  So an element in the quotient space can be  represented by the orbit $\sigma(X)=XR$ where $R\in O(k)$. Then an embedding can be given by
\begin{align*}
J(\sigma(X))=XX^T.
\end{align*}
The collection of $XX^T$ forms a subspace of $\R^{m^2}$.  We now verify that $J$ is an equivariant embedding under the group action of $G=O(m)$.
Letting $g\in G=O(m)$, one has $J(gX)=gXX^Tg^T=\phi(g)J(X),$ where the map $\phi(g)=g$ acts on the image $J(X)$ by the conjugation map. That is, $\phi(g)J(X)=gXX^Tg^T.$

Given the estimate  $\widehat{F}(x)$, the next step is to derive the projection of  $\widehat{F}(x)$ onto $\widetilde M=J(M)$. Since all  $XX^T$ form a subspace, one can use the following procedure to calculate the map from  $\widehat{F}(x)$ to the Grassmann manifold by finding an orthonormal basis for the image. This algorithm is a special case of the projection via Conway embedding \citep{brian2014}.

\begin{enumerate}
\item Find the eigendecomposition $\widehat{F}(x) = Q\Lambda Q^{-1}$
\item Take the $k$ eigenvectors corresponding to the top $k$ eigenvalues in $\Lambda$ as an orthonormal basis for $\widehat{F}_E(x)$, $Q_{[1:k,]}$.
\end{enumerate}

We now consider two illustrative examples, one synthetic and one from a financial time series, for extrinsic kernel regression with subspace response variables. The technique is unique compared to other subspace regression techniques because the extrinsic distance offers a well defined and principled distance between responses of different dimension. This prevents having to constrain the responses to be a fixed dimension or hard coding a heuristic distance between subspaces of different dimension into the distance function.

We now consider a synthetic example in which the predictors are the time points and the responses are points on the Grassmann manifold. Since we represent subspaces with draws from the Stiefel manifold, we draw orthonormal bases from the Matrix von Mises-Fisher distribution as their representation. We generate $N$ draws from the following process with concentration parameter $\kappa$, in which the first $n_1$ draws are of dimension  $4$ and the last $n_2$ draws are of  dimension $5$,
\begin{algorithmic}
\For{$1 \leq t \leq N$}
	\State Draw $X \sim MN(0,I_m, I_{5})$
	\State  $\mu_{[,1]} := t + X_{[,1]}$, \;$\mu_{[,2]} := t - X_{[,2]}$,\; $\mu_{[,3]} := t^2 + X_{[,3]}$,\; $\mu_{[,4]} := tX_{[,4]}$
		\If{$t > n_1$}
			\State $\mu_{[,5]} := t + tX_{[,5]}$
		\EndIf
	\State $Y_t := \operatorname{vMF}(\kappa M)$
\EndFor
\end{algorithmic}
Here the only covariate associated with $Y_t$ is $t$. With a concentration of $\kappa = 1$,  and $n_1 = n_2 = 50$, we generate much noisier data than before, and are able to correctly predict the dimension of the subspace at each time point. When examining the pairwise distance between the realizations in Figure ~\ref{fig:synthdistmat},  it is clear that the extrinsic distance distinguishes between dimensions and does not require any specification of the dimension. The predicted dimension at each time point and the residuals are plotted in Figure ~\ref{fig:synthpred}.

\begin{figure}[ht]
\begin{center}
\includegraphics[scale=.8]{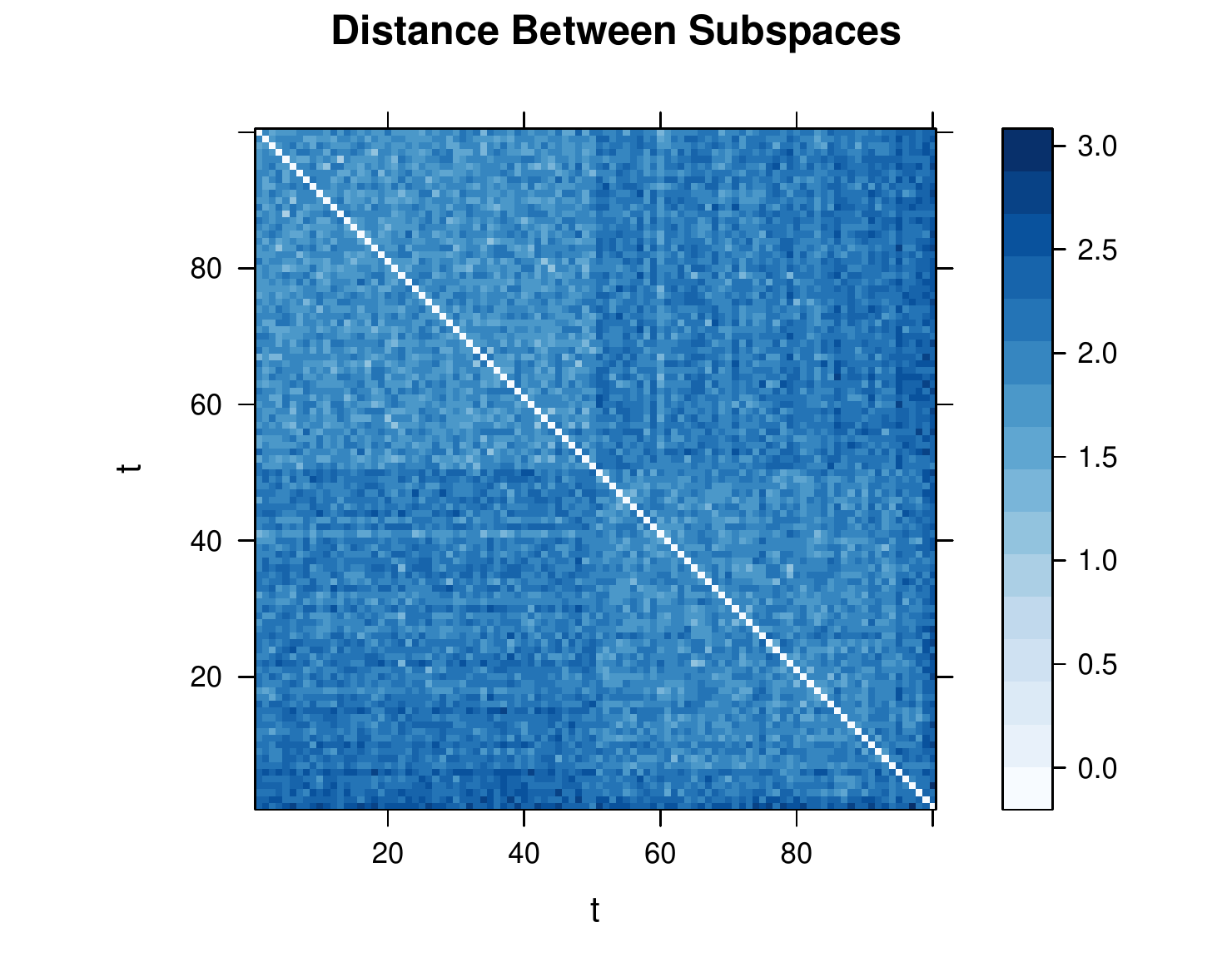}
\end{center}
\caption{Pairwise distance between two observations generated by the specified model indexed by $t$ measured by distance between points in the Conway embedding. This visualization of the extrinsic distance shows the cluster by dimension.}
\label{fig:synthdistmat}
\end{figure}

\begin{figure}[ht]
\begin{center}
\includegraphics[scale=.5]{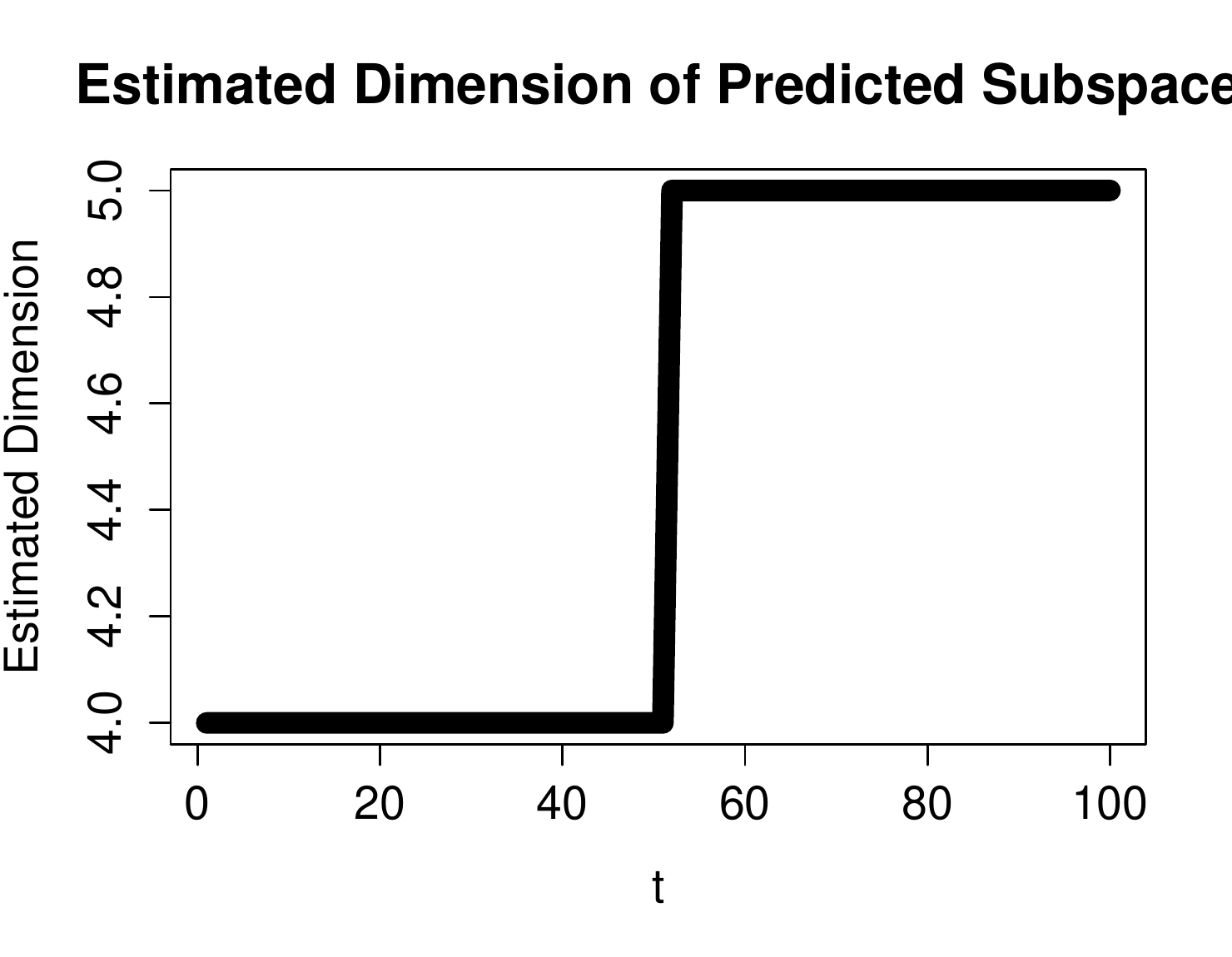}
\includegraphics[scale=.5]{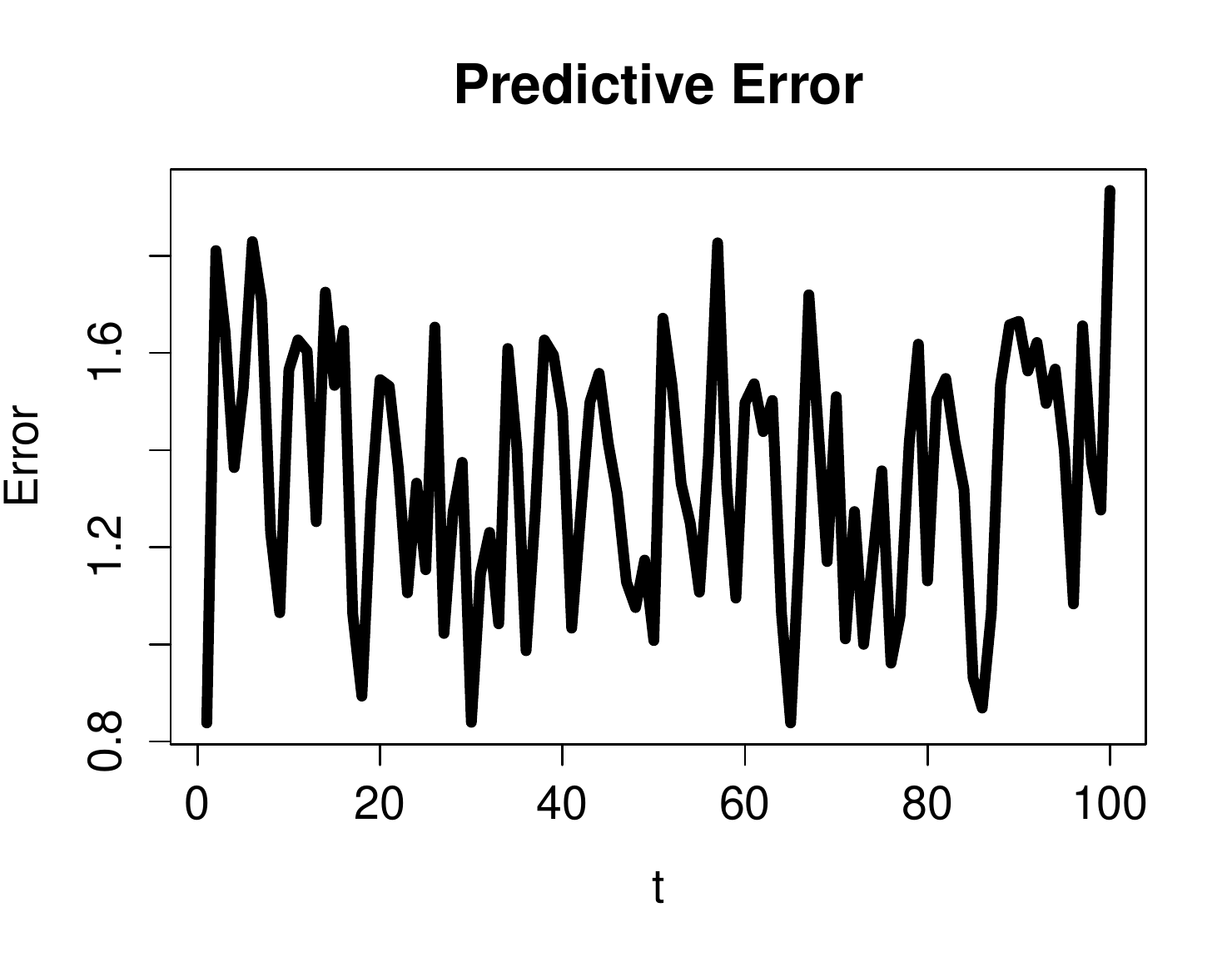}
\end{center}
\caption{The estimated dimension and residual for the extrinsic kernel regression estimate at each time point $t$ from data generated from the specified model. The regression estimate is accurate on the dimension of the subspace and prediction residuals are consistent with a concentration parameter $\kappa = 1$.}
\label{fig:synthpred}
\end{figure}

The key advantage of this method is not requiring any constraints on the dimension of the input or output subspaces. This is important in some examples, such as high dimensional time series analysis with data such as high frequency trading where the analysis usually culminates in analyzing principal components, or eigenvectors of the large covariance matrix estimated between assets. Market events can change asset covariances which in turn changes the number of significant eigenvectors, so a method automatically interpolating time points must not depend on specifying the number of significant eigenvectors.

We apply this method to the Istanbul Stock Exchange on UCI Machine Learning Repository \cite{financedata}, using the 5 index funds S\&P 500 (SP), the Istanbul Stock Exchange (ISE), stock market return index of Japan (NIKKEI), MSCI European index (EU), and the stock market return index of Brazil (BOVESPA). The data contain 97 full weeks over 115 weeks of daily market closing values from January 5, 2009 to February 18, 2011. For each week, a covariance matrix is estimated between the assets, of which the eigenvectors with eigenvalues greater than $10^{-10}$ are retained as the orthonormal basis corresponding to the covariance matrix. As can be seen in Figure ~\ref{fig:finance_covars}, these covariance matrices change significantly over time. These orthonormal matrices are given to the model along with the corresponding week and each week is predicted.

\begin{figure}[ht]
\begin{center}
\includegraphics[scale=.4]{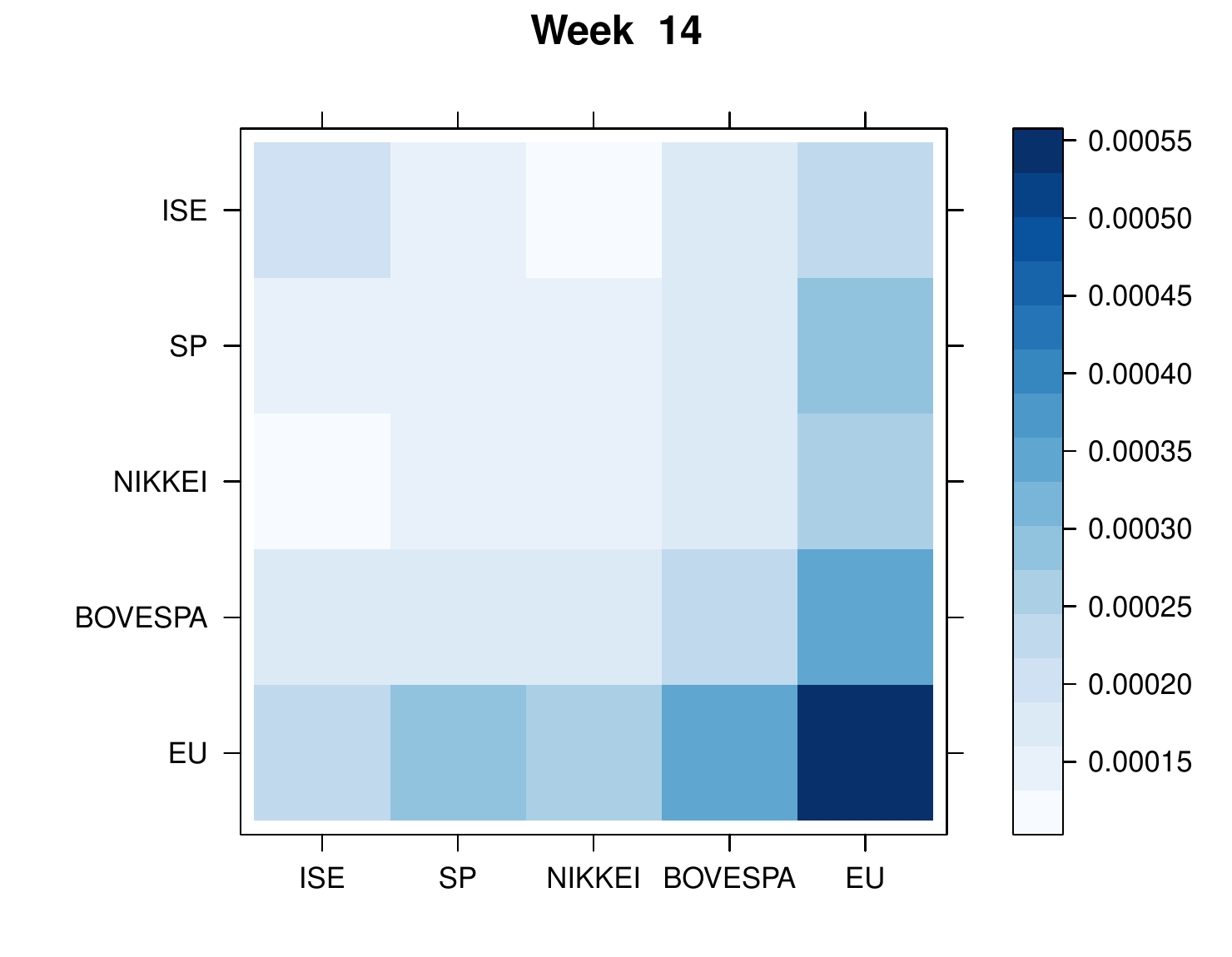}
\includegraphics[scale=.4]{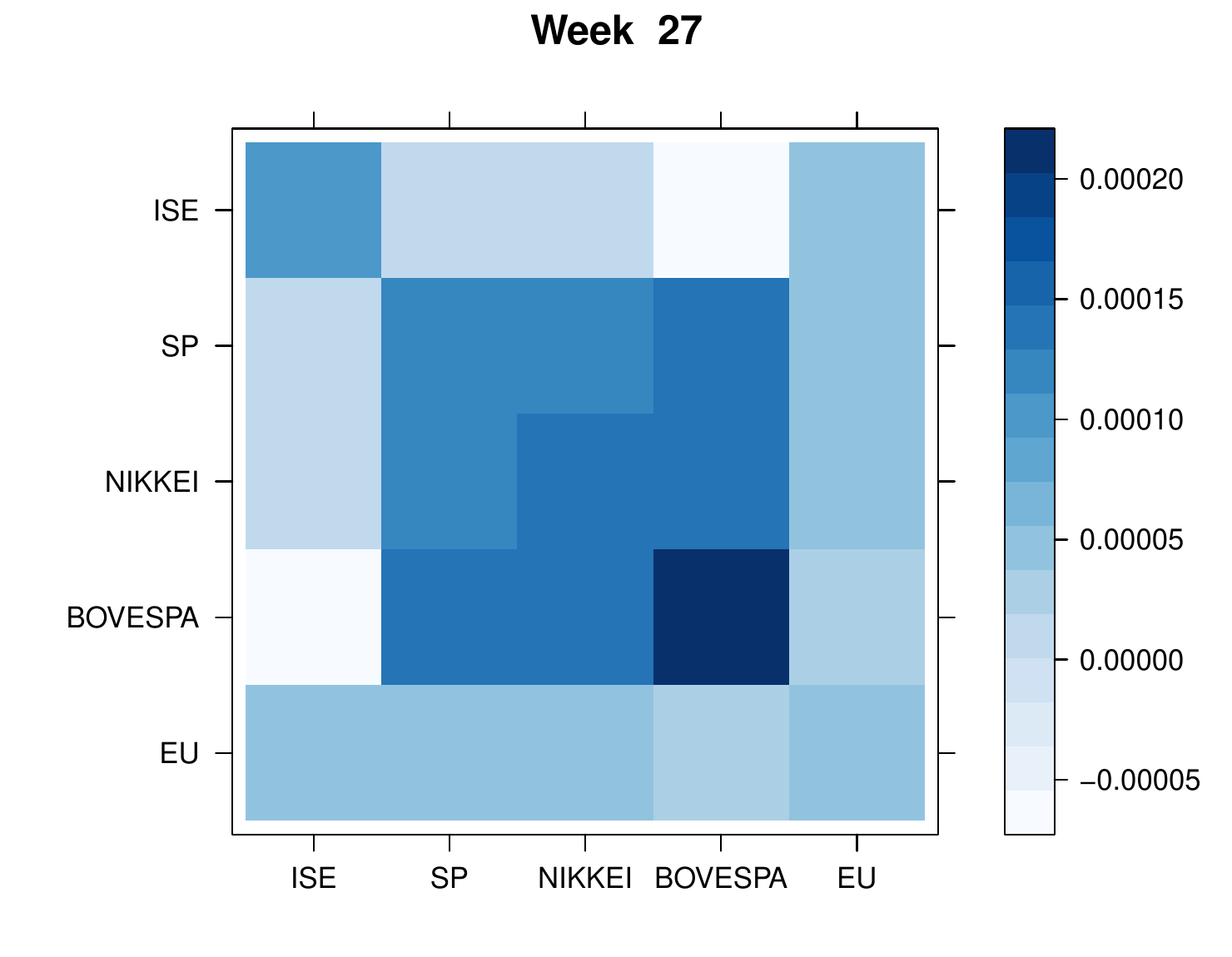}
\includegraphics[scale=.4]{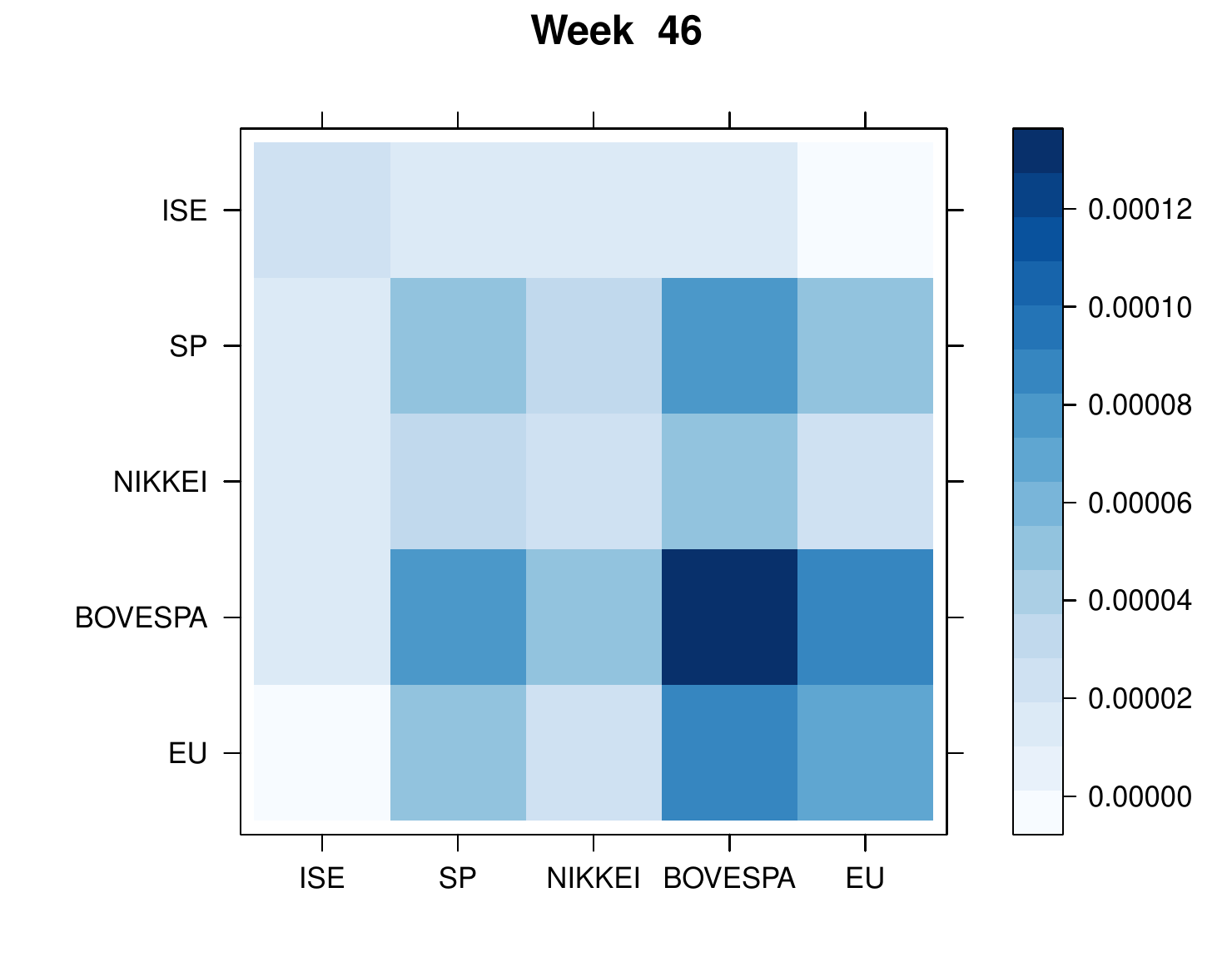}
\includegraphics[scale=.4]{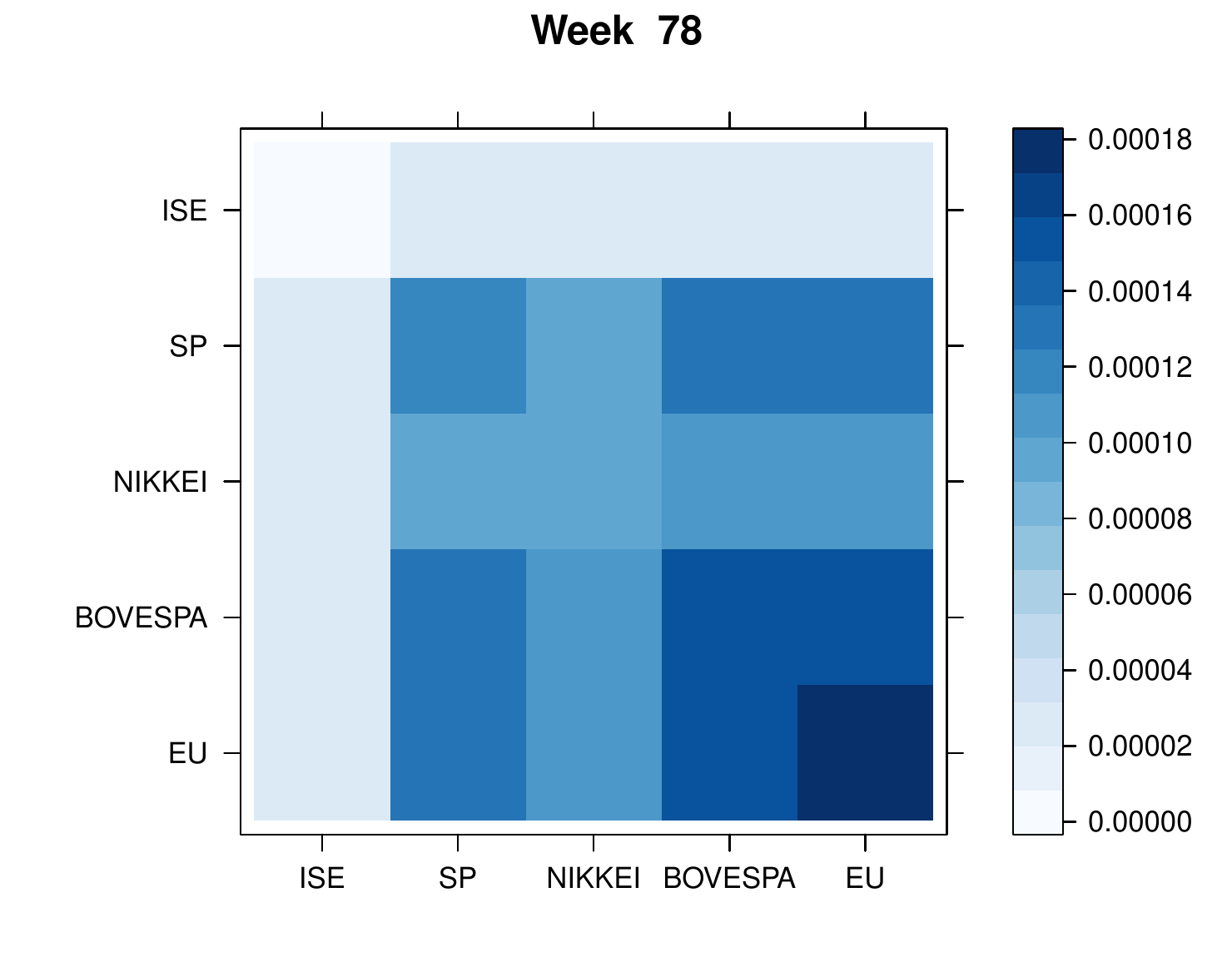}
\end{center}
\caption{The covariance matrices estimated from the daily closing prices from various stock market index funds over various weeks. Covariance between markets change substantially from week to week.}
\label{fig:finance_covars}
\end{figure}

The residuals are shown in Figure ~\ref{fig:financeresids} compared to the observed distance between subspaces between two consecutive observations. The method is predicting the subspaces within the variance of data, which means there is information or at least structure to how the covariance matrices are evolving over time -- the relationships are not purely random week to week.

\begin{figure}[ht]
\begin{center}
\includegraphics[scale=.7]{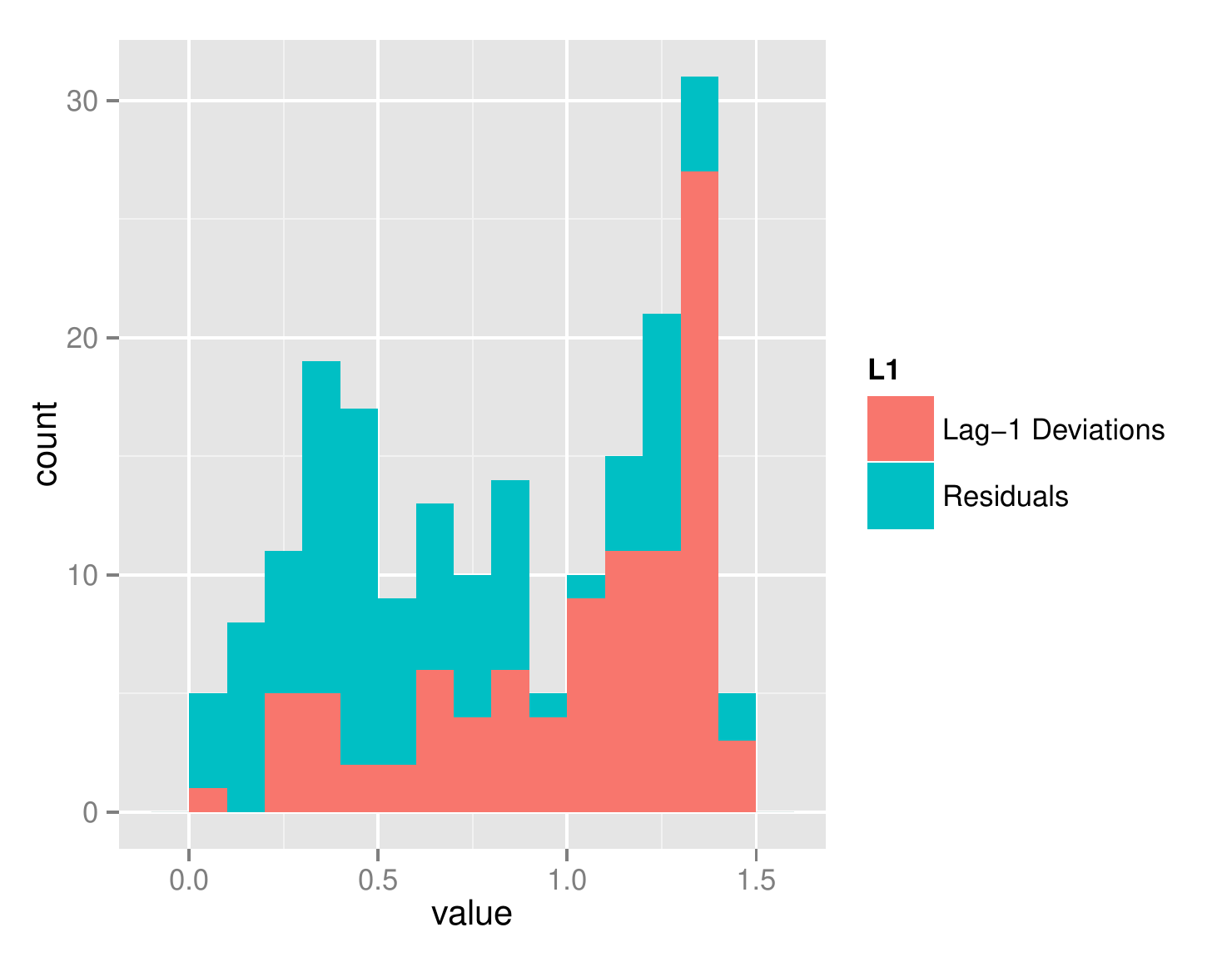}
\end{center}
\caption{The distribution of residuals from extrinsic kernel regression compared to the distribution of the distance between observations. The in general smaller residuals suggest that the changes in covariance structure is dynamic in a learnable way.}
\label{fig:financeresids}
\end{figure}


\end{example}

\section{Asymptotic properties of the extrinsic regression model}
\label{sec-th}

In this section, we investigate the large sample properties of our extrinsic regression estimates. We assume the marginal density $f_X(x)$ is differentiable and  the absolute value of any of the partial derivatives  of $f_X(x)$ of order two are bounded by some constant $C$. In our proof, we assume our kernel function $K$ takes a product form. That is, $K(x)=K_1(x^1)\cdots K_m(x^m)$ where $x=(x^1,\ldots, x^m)$ and $K_1,\ldots, K_m$ are one dimensional symmetric kernels such that $\int_{\mathbb R} K_i(u)du=1$,  $\int_{\mathbb R} uK_i(u)du=0$ and $\int_{\mathbb R} u^2K_i(u)du<\infty$ for $i=1,\ldots, m$. The results can be  generalized to kernels with arbitrary form and with $H$ given by a more general positive definite matrix instead of a diagonal matrix. Theorem \ref{th-ker} derives  the asymptotic distribution of the extrinsic regression estimate $\widehat F_E(x)$ for any $x$. 

\begin{theorem}
\label{th-ker}
Let $\mu(x)=E\left(\widetilde{P}(dy|x)\right)$, which is the conditional mean regression function of $\widetilde P$ and assume $\mu(x)$ is differentiable.  Assume $n|H|\rightarrow \infty$.  Denote $x=(x^1,\ldots, x^m)$.  Let $\widetilde\mu(x)=\mu(x)+\dfrac{Z(x)}{f_X(x)}$, where the $i$th component  $Z_i(x)$ ($i=1,\ldots, D$) of $Z(x)$  is given by  
\begin{align}
\label{eq-bias1}
Z_i(x)=&\nonumber h_1^2\left( \dfrac{\partial f}{\partial x^1}\dfrac{\partial \mu_i}{\partial x^1}+\dfrac{1}{2}f_X(x) \left(\dfrac{\partial^2 \mu_i}{\partial (x^1)^2}+\ldots+ \dfrac{\partial^2 \mu_i}{\partial x^mx^1} \right) \right)\int v_1^2K_1(v_1)dv_1+\ldots\\
&+ h_m^2\left( \dfrac{\partial f}{\partial x^m}\dfrac{\partial \mu_i}{\partial x^m}+   \dfrac{1}{2}f_X(x)\left( \dfrac{\partial^2 \mu_i}{\partial x^1x^m}+\ldots+ \dfrac{\partial^2 \mu_i}{\partial (x^m)^2}\right)  \right)\int v_m^2K_m(v_m)dv_m.
\end{align}  Assume the projection $\mathcal P$ of $\widetilde\mu(x)$ onto $\widetilde{M}=J(M)$ is unique and $\mathcal P$ is continuously differentiable in a neighborhood of $\widetilde\mu(x)$. Then the following holds assuming $P(dy\mid x)\circ J^{-1}$ has finite second moments:

\begin{align}
\label{th-consist}
\sqrt{n|H|}d_{\widetilde\mu(x)} \mathcal P\left(\widehat{F}(x)-\widetilde{\mu}(x)\right)\xrightarrow{L} N(0,\widetilde\Sigma(x)),
\end{align}
 where $d_{\widetilde{\mu}(x)}\mathcal P$ is the differential   from $T_{\widetilde\mu(x)}\R^D$ to $T_{\mathcal P(\widetilde\mu(x))}\widetilde{M}$ of the projection map $\mathcal{P}$ at  $\widetilde{\mu}(x)=\mu(x)+\frac{Z(x)}{f_X(x)}$.  Here $\widetilde\Sigma(x)=B^T\bar \Sigma(x) B$, where $B$ is the $D\times d$ matrix of the differential $d_{\widetilde\mu(x)}\mathcal P$ with respect to  given  orthonormal bases of $T_{\widetilde\mu(x)}\R^D$ and $T_{\mathcal P(\widetilde\mu(x))}\widetilde{M}$,  and the $(j,k)$th entry of $\bar \Sigma(x)$ is given by \eqref{eq-sigbar} with
\begin{align}\bar \Sigma_{jk}=\dfrac{\sigma(J_j(y), J_k(y))\int K(v)^2dv}{f_X(x)},
\end{align} where $\sigma(J_j(y), J_k(y))=\cov(J_j, J_k)$, and $J_j$ is the $j$th element of $J(y)$.  Here $\xrightarrow{L}$ indicates convergence in distribution.

\end{theorem}

Corollary \ref{coro-mise} is on the mean integrated squared error of the estimates. 

\begin{corollary}
\label{coro-mise}
Assuming  the same conditions of Theorem \ref{th-ker} and  the covariate space is bounded, the mean integrated squared error of $\widehat{F}_E(x)$ is of the order $O(n^{-4/(m+4)})$, with the choice of $h_i$'s ($i=1,\ldots, m$) to be of the same order, that is, of $O(n^{-1/(m+4)})$.
\end{corollary}

\begin{remark}
Note that in nonparametric regression with both predictors ($m$-dimensional) and responses in the Euclidean space, the optimal order of the mean integrated squared error is $O(n^{-4/(m+4)})$ under the assumption that the true regression function has bounded second derivative. Our method achieves the same rates. However,  whether such rates are minimax in the context of manifold valued response is not known.
\end{remark}

Theorem  \ref{th-uniform} shows some results on uniform convergence rates of the estimator.
\begin{theorem}
\label{th-uniform}
Assume the covariate space $x \in \mathcal X\subset \R^m$ is compact and $\mathcal P$ has continuous first derivative.  Then 
\begin{align}
\label{eq-uniform}
\sup_{x\in \mathcal{X}}\|d_{\widetilde\mu(x)}\mathcal{P} \left ( \widehat{F}(x)-E(\widehat{F}(x))\right)\|=O_p\left(\log^{1/2} n/\sqrt{n|H|}\right).
\end{align}
\end{theorem}

As pointed out in Remark \ref{re-poly}, it is ideal in many cases  to fit a higher order (say $p$th order)  local polynomial model in estimating $\mu(x)$ before projecting back onto the image of the manifold. Such estimates are more appealing especially when $F(x)$ is more curved over a neighborhood of $x$. One can show that similar results as those of Theorem \ref{th-ker} hold, though with much more involved argument.

We now give details of such estimators and their asymptotic distributions are derived in Theorem \ref{th-poly}.
Recall $F(x)=E\left(P\left(dy\mid x\right)\right)$ and  $\mu(x)=E\left(\widetilde P(dy\mid x)\right)$ and $J(y_1),\ldots, J(y_n)$ are the points on $\widetilde {M}=J(M)$ after embedding $J$. We first obtain an estimate $\widehat{F}(x)$ of $\mu(x)$ using  $p$th order local polynomials estimation. The intermediate estimate   $\widehat{F}(x)$  is then projected back to $\widetilde M$ serving as the ultimate estimate of $F(x)$. The general framework is given as follows:
\begin{align}
\label{eq-pthlocalpoly}
\{\hat{\beta}^j_{\bk}(x)\}_{0\leq |\bk|\leq p,\;1\leq j\leq D}& \\ \nonumber
=\argmin_{\{\beta^j_{\bk}(x)\}_{0\leq |\bk|\leq p,\;1\leq j\leq D} }\sum_{i=1}^n \Big(&\big\|J(y_i)-\big(\sum_{0\leq |\bk|\leq p}\beta^1_{\bk}(x)(x_i-x)^{|\bk|},\ldots,\sum_{0\leq |\bk|\leq p}\beta^D_{\bk}(x)(x_i-x)^{|\bk|} \big)^T\big\|^2\\
&\times K_H(x_i-x)\Big).
\end{align}
Some of the notation used in \eqref{eq-pthlocalpoly} are given as follows:
\begin{align*}
\bk=(k_1,\ldots, k_m),\; |\bk|=\sum_{l=1}^m k_l,\;|\bk|\in \{0,\ldots, p\}, \\
\bk!=k_1!\times\ldots\times k_m!,\; x^{\bk}=(x^1)^{k_1}\times\ldots\times (x^m)^{k_m}\\
\sum_{0\leq |\bk|\leq p}=\underset{|\bk|=k_1+\ldots+k_m=j}{\sum_{j=0}^{p}\sum_{k_1=0}^j\ldots\sum_{k_m=0}^j}.
\end{align*}
When $\bk$=0, $\left(\widehat\beta_{\boldsymbol 0}^1,\ldots, \widehat\beta_{\boldsymbol 0}^D\right)^T$ corresponds to the kernel estimator, which is the same as  the estimator given in \eqref{eq-kesti}. When $p=1$, $\left(\widehat\beta_{ \bk=0}^1,\ldots, \widehat\beta_{\bk=0}^D\right)^T$ coincides with the estimator $\widehat{\boldsymbol{\beta}}_0$ in \eqref{eq-localinear}.

Finally, we have
\begin{align}
\widehat{F}(x)&=\hat{\boldsymbol\beta}_0(x)=\left(\widehat\beta_{ \bk=0}^1,\ldots, \widehat\beta_{\bk=0}^D\right)^T\label{eq-polyesti2},\\
\widehat{F}_E(x)&=J^{-1}\left(\mathcal P(\widehat{F}(x))\right)=J^{-1}\left(\argmin_{q\in \widetilde{M}}||q-\widehat{F}(x)||\right) \label{eq-polyesti3}.
\end{align}

Theorem \ref{th-poly} derives the asymptotic distribution of $\widehat{F}_E(x)$, with $\widehat{F}(x)$ obtained using $p$th order polynomials local regression of $J(y_1),\ldots, J(y_n)$ given in \eqref{eq-polyesti2}.

\begin{theorem}
\label{th-poly}
 Let $\widehat{F}_E(x)$ be given in \eqref{eq-polyesti3}. Assume the $(p+2)$th moment of the kernel function $K(x)$ exists and  $\mu(x)$ is ($p+2$)th order differentiable in a neighborhood of $x=(x^1,\ldots, x^m)$. Assume the projection $\mathcal P$ of $\widetilde \mu(x)$ onto $\widetilde{M}=J(M)$ is unique and $\mathcal P$ is continuously differentiable in a neighborhood of $\widetilde{\mu}(x)$, where  $\widetilde{\mu}(x)=\mu(x)+\Bias(x)$, with $\Bias(x)$ given in \eqref{eq-polybias}. If  $P(dy\mid x)\circ J^{-1}$ has finite second moments, then we have:

 \begin{align}
\label{th-consist2}
\sqrt{n|H|}d_{\widetilde{\mu}(x)} \mathcal P\left(\widehat{F}(x)-\widetilde{\mu}(x)\right)\xrightarrow{L} N(0,\widetilde\Sigma(x)),
\end{align}
 where $d_{\widetilde{\mu}(x)}\mathcal P$ is the differential  from $T_{\widetilde{\mu}(x)}\R^D$ to $T_{\mathcal P\widetilde{\mu}(x)}\widetilde M$ of the projection map $\mathcal{P}$ at  $\widetilde{\mu}(x)$.  Here $\Sigma(x)=B^T\bar \Sigma(x) B$, where $B$ is the $D\times d$ matrix of the differential $d_{\widetilde{\mu}(x)}\mathcal P$ with respect to given orthonormal basis of tangent space $T_{\widetilde{\mu}(x)}\R^D$ and tangent space $T_{\mathcal P\widetilde{\mu}(x)}\widetilde M$ and the $jk$th entry of $\bar \Sigma(x)$ is given by \eqref{eq-sigbar2}.  Here $\xrightarrow{L}$ indicates convergence in distribution.
\end{theorem}

\begin{remark}
Note that the order of the bias term $\Bias(x)$ (given in \eqref{eq-polybias}) differs when $p$ is even (see \eqref{eq-polybias1}) and when $p$ is odd (see \eqref{eq-polybias2}).
\end{remark}

\section{Conclusion}

We have proposed an extrinsic regression framework for modeling data with manifold valued responses and shown desirable asymptotic properties of the resulting estimators. We applied this framework to a variety of applications, such as  responses restricted to the sphere, shape spaces, and linear subspaces. The principle motivating this framework is that kernel regression and Riemannian geometry both rely on locally Euclidean structures. This property allows us to construct inexpensive estimators without loss of predictive accuracy as demonstrated by the asymptotic behavior of the mean integrated square error, and also the empirical results. Empirical results even suggest that the extrinsic estimators may perform better due to their reduced complexity and ease of optimizing tuning parameters such as kernel bandwidth. Future work may also use this principle to guide sampling methodology when trying to sample parameters from a manifold or optimizing an EM-algorithm, where it may be computationally or mathematically difficult to restrict intermediate steps to the manifold.


\section*{Appendix}

\begin{proof}[Proof of Theorem \ref{th-ker}]

 Recall
\begin{equation*}
\widehat{F}(x)=\dfrac{\frac{1}{n }\sum_{i=1}^nJ(y_i)K_H(x_i-x)}{\frac{1}{n }\sum_{i=1}^nK_H(x_i-x)} .
\end{equation*}
Denote  the denominator of $\widehat{F}(x)$ as
$$\widehat{f}(x)=\frac{1}{n }\sum_{i=1}^nK_H(x_i-x)=\frac{1}{n\mid H\mid  }\sum_{i=1}^nK(x_i-x).$$
It is standard to show 
\begin{align}
\label{eq-kerD}
\widehat{f}(x)\xrightarrow{P}f_X(x)
\end{align}
where $\xrightarrow{P}$ indicates convergence in probability.
For the numerator term of $\widehat{F}(x)$, one has
\begin{align*}
E\left(\frac{1}{n }\sum_{i=1}^nJ(y_i)K_H(x_i-x))\right)&=\frac{1}{n }\sum_{i=1}^n E\left(J(y_i)K_H(x_i-x))  \right)\\
&=\frac{1}{n }\sum_{i=1}^n  \int E\left( J(y_i)K_H(x_i-x)) \mid x_i \right) f_X(x_i)dx_i\\
&=\frac{1}{n }\sum_{i=1}^n  \int \mu(x_i)K_H(x_i-x)) f_X(x_i)dx_i\\
&=\int \mu(\widetilde x )K_H(\widetilde x-x)) f_X(\widetilde x)d\widetilde x.
\end{align*}
Noting that $\mu(x)=(\mu_1(x),\ldots, \mu_D(x))'\in \mathbb R^D$, we slightly abuse  the integral notation above meaning that the $j$th entry of $E\left(n^{-1}\sum_{i=1}^nJ(y_i)K_H(x_i-x))\right)$ is given by
\begin{align*}
\int\mu_j(\widetilde x )K_H(\widetilde x-x)) f_X(\widetilde x)d\widetilde x.
\end{align*}

Letting $v= H^{-1}(\widetilde x-x)$ by changing of variables, the above equations become
\begin{align*}
E\left(\frac{1}{n }\sum_{i=1}^nJ(y_i)K_H(x_i-x))\right)&=\int \mu(x+Hv )K(v) f_X(x+Hv)dv.
\end{align*}
By the multivariate Taylor expansion,
\begin{align}
\label{eq-taylorf}
f_X(x+Hv)=f_X(x)+(\bigtriangledown f)\cdot (Hv)+R,
\end{align}
where  $\bigtriangledown f$ is the gradient of $f$ and $R$ is the remainder term  of the expansion. The remainder  $R$ can be shown to be bounded above by
\begin{align*}
R\leq \frac{C}{2} \| Hv\|^2, \; \| Hv\|=|h_1v_1|+\ldots |h_mv_m|.
\end{align*}
Note that $\mu(x+Hv ) $ is a multivariate map valued  in $\mathbb R^D$.  We can make  second order multivariate Taylor expansions for $\mu(x+Hv )=(\mu_1(x+Hv),\ldots, \mu_D(x+Hv))' $ at each of its entries $\mu_i$ for $i=1,\ldots, D$. We have
\begin{align}
\label{eq-taylormu}
\mu(x+Hv ) =\mu(x)+A(Hv)+V+R,
\end{align}
where $A$ is a $D\times m$ matrix whose $i$th row is given by the gradient of $\mu_i$ evaluated at $x$. $V$ is a $D$-dimensional vector, whose $i$th term is given by $\frac{1}{2}(Hv)^tT_i(Hv)$, where $T_i$ is the Hessian matrix of $\mu_i(x)$ and  $R$ is the remainder vector.
Thus,
\begin{align}
&E\left(\frac{1}{n }\sum_{i=1}^nJ(y_i)K_H(x_i-x))\right)\\  \nonumber
&\approx \int \left(( f_X(x)+(\bigtriangledown f)\cdot (Hv)) K(v)(\mu(x)+A(Hv)+V)\right)dv\\ 
 &=f_X(x)\mu(x)+ f_X(x)\int K(v)A(Hv)dv+f_X(x)\int K(v)Vdv\label{eq-2.10} \\
&+\mu(x)\int (\bigtriangledown f)\cdot (Hv) K(v)dv +\int (\bigtriangledown f)\cdot (Hv) K(v) A(Hv)dv+\int (\bigtriangledown f)\cdot (Hv) K(v)Vdv \label{eq-2.11}. \\ \nonumber
\end{align}
By the property of the kernel function, we have  $\int K(u)udu=0$;  therefore the  second term of equation \eqref{eq-2.10} is zero by simple algebra.
To evaluate the third term of equation \eqref{eq-2.10},  we first calculate for $\int K(v)Vdv$. From here onward until the end of the proof, we denote $x=(x^1,\ldots, x^m)$ where $x^i$ is the $i$th coordinate of $x$.  Note that the $i$th term of $V$ ($i=1,\ldots, D$) is given by  $\dfrac{1}{2}(Hv)^tT_i(Hv)$, where $T_i$ is the Hessian matrix of $\mu_i$, which is precisely
\begin{align*}
\dfrac{1}{2}h_1^2v_1^2 \left(\dfrac{\partial^2 \mu_i}{\partial (x^1)^2}+\ldots+ \dfrac{\partial^2 \mu_i}{\partial x^mx^1} \right)+\ldots+\frac{1}{2}h_m^2v_m^2\left(  \dfrac{\partial^2 \mu_i}{\partial x^1x^m}+\ldots+ \dfrac{\partial^2 \mu_i}{\partial (x^m)^2}\right).
\end{align*}
Therefore, the $i$th entry of the third term of  equation \eqref{eq-2.10} is given by
\begin{align}
\label{eq-U}
U_i&=\dfrac{1}{2}f_X(x) \Big(h_1^2 \left(\dfrac{\partial^2 \mu_i}{\partial (x^1)^2}+\ldots+ \dfrac{\partial^2 \mu_i}{\partial x^mx^1} \right)\int v_1^2K_1(v_1)dv_1+\ldots\\ \nonumber
&+h_m^2\left(  \dfrac{\partial^2 \mu_i}{\partial x^1x^m}+\ldots+ \dfrac{\partial^2 \mu_i}{\partial (x^m)^2}\right)\int v_m^2K_m(v_m)dv_m\Big).
\end{align}
The first term of  equation \eqref{eq-2.11} is given by
\begin{align*}
\mu(x)\int (\bigtriangledown f)\cdot (Hv) K(v)dv=\int\left(h_1v_1\dfrac{\partial f}{\partial x^1}+\ldots+h_mv_m\dfrac{\partial f}{\partial x^m}\right)K(v)dv=0.
\end{align*}
The $i$th entry of the second term of equation \eqref{eq-2.11} is given by
\begin{align}
\label{eq-Z}
h_1^2\dfrac{\partial f}{\partial x^1}\dfrac{\partial \mu_i}{\partial x^1}\int v_1^2K_1(v_1)dv_1+\ldots+h_m^2\dfrac{\partial f}{\partial x^m}\dfrac{\partial \mu_i}{\partial x^m}\int v_m^2K_m(v_m)dv_m.
\end{align}

The third term of  equation \eqref{eq-2.11} can be shown to be zero, since odd moments of
symmetric kernels are 0.
Therefore, we have
\begin{align}
\label{eq-bias0}
E\left(\frac{1}{n }\sum_{i=1}^nJ(y_i)K_H(x_i-x))\right)&\approx f_X(x)\mu(x)+Z,
\end{align}
where  the $i$th coordinate of $Z$ is
\begin{align}
\label{eq-bias}
Z_i=&\nonumber h_1^2\left\{ \dfrac{\partial f}{\partial x^1}\dfrac{\partial \mu_i}{\partial x^1}+\dfrac{1}{2}f_X(x) \left(\dfrac{\partial^2 \mu_i}{\partial (x^1)^2}+\ldots+ \dfrac{\partial^2 \mu_i}{\partial x^mx^1} \right) \right\}\int v_1^2K_1(v_1)dv_1\\ \nonumber
&+\ldots\\
&+ h_m^2\left\{ \dfrac{\partial f}{\partial x^m}\dfrac{\partial \mu_i}{\partial x^m}+   \dfrac{1}{2}f_X(x)\left( \dfrac{\partial^2 \mu_i}{\partial x^1x^m}+\ldots+ \dfrac{\partial^2 \mu_i}{\partial (x^m)^2}\right)  \right\}\int v_m^2K_m(v_m)dv_m
\end{align}
combining equations \eqref{eq-U} and  \eqref{eq-Z}. The reminder term of \eqref{eq-taylorf} is of order $o(\max \{h_1,\ldots, h_m \})$  and each entry of the remainder vector in \eqref{eq-taylormu} is of order $o(\max \{h_1^2,\ldots, h_m^2 \})$.

We now look at the covariance matrix of $n^{-1}\sum_{i=1}^nJ(y_i)K_H(x_i-x))$,  which we denote by $\Sigma(x)$.  Denote the $j$th entry ($j=1,\ldots, D$) of $J(y_i)$  as $J_j(y_i)$.  
Denote $\sigma(y^j, y^k)$ as the conditional covariance between the $i$th entry  and $j$th entry of $y$. We have
\begin{align*}
\Sigma_{jk}=E\Big[  & \left(\frac{1}{n }\sum_{i=1}^nJ_j(y_i)K_H(x_i-x))-E\left(\frac{1}{n }\sum_{i=1}^nJ_j(y_i)K_H(x_i-x))\right)  \right)\\
&\left(\frac{1}{n }\sum_{i=1}^nJ_k(y_i)K_H(x_i-x))-E\left(\frac{1}{n }\sum_{i=1}^nJ_k(y_i)K_H(x_i-x)) \right)  \right)\Big]\\
=E\Big[  & \left(\frac{1}{n }\sum_{i=1}^n\left(J_j(y_i)K_H(x_i-x))-\int \mu_j(\widetilde{x})K_H(\widetilde x-x)f_X(\widetilde x)d\widetilde x\right)\right)\\
&\left(\frac{1}{n }\sum_{i=1}^n\left(J_k(y_i)K_H(x_i-x))-\int \mu_k(\widetilde{x})K_H(\widetilde x-x)f_X(\widetilde x)d\widetilde x  \right)\right)\Big]\\
=\frac{1}{n}\int E\Big[  & \left(J_j(y_1)K_H(x_1-x))-\int \mu_j(\widetilde{x})K_H(\widetilde x-x)f_X(\widetilde x)d\widetilde x\right)\\
&\left(J_k(y_1)K_H(x_1-x))-\int \mu_k(\widetilde{x})K_H(\widetilde x-x)f_X(\widetilde x)d\widetilde x  \right)\mid x_1\Big]f_X( x_1)dx_1\\
&=\frac{1}{n}\int \sigma(J_j(y_1)K_H(x_1-x)), J_k(y_1)K_H(x_1-x))f_X(x_1)dx_1\\
&=\frac{1}{n}\int K_H(x_1-x))^2\sigma(J_j(y_1), J_k(y_1))f_X(x_1)dx_1.
\end{align*}
By the change of  variable $v=H^{-1}(x_1-x)$, the above equation becomes
\begin{align}
\label{eq-cov}
\Sigma_{jk}&\nonumber=\frac{1}{n| H|}\int K(v)^2\sigma(J_j(y_{v}), J_k(y_{v}))f_X(Hv+x)dv\\ \nonumber
&=\frac{1}{n | H|}\int K(v)^2\sigma(J_j(y_{v}), J_k(y_{v}))\left( f_X(x)+\bigtriangledown f\cdot (Hv)+o(\max\{ h_1,\ldots, h_m \})\right)dv\\
&=\frac{1}{n| H|}\int K(v)^2\sigma(J_j(y_{v}), J_k(y_{v}))f_X(x))dv+o\left(\frac{1}{n|H| }\right) .
\end{align}
By \eqref{eq-kerD}, \eqref{eq-bias0} and  \eqref{eq-cov}, and applying central limit theorem and Slustky's theorem,  one has
\begin{align}
\label{eq-cltnoproj}
\sqrt{n|H|} \left (\widehat{F}(x)-\widetilde{\mu}(x)\right) \xrightarrow{L} N(0,\bar \Sigma(x)),
\end{align}
where $\widetilde{\mu}(x)=\mu(x)+\frac{Z}{f_X(x)} $ and the $i$th entry ($i=1,\ldots, D$) of $Z$ is given by \eqref{eq-bias} and
\begin{align}
\label{eq-sigbar}
\bar \Sigma_{jk}=\dfrac{\sigma(J_j(y_{v}), J_k(y_{v}))\int K(v)^2dv}{f_X(x)}.
\end{align}
One can show
\begin{align*}
\sqrt{n|H|}\left(\widehat{F}_E(x)-\mathcal P \left(\widetilde{\mu}(x)\right)\right)=\sqrt{n|H|}d_{\widetilde\mu(x)}\mathcal P\left (\widehat{F}(x)- \widetilde\mu(x)\right)+o_P(1).
\end{align*}
Therefore, one has
\begin{align}
\sqrt{n|H|}d_{\widetilde\mu(x)}\mathcal P\left (\widehat{F}(x)- \widetilde{\mu}(x)\right) \xrightarrow{L} N(0,\widetilde \Sigma(x)).
\end{align}
Here $\widetilde\Sigma(x)=B^T\bar \Sigma(x) B$,  where $B$ is the $D\times d$ matrix of the differential $d_{\widetilde{\mu}(x)}\mathcal P$ with respect to given orthonormal bases of $T_{\widetilde{\mu}(x)}\R^D$ and $T_{\mathcal P\widetilde\mu(x)}\widetilde{M}$ .

\end{proof}

\begin{proof}[Proof of Corollary \ref{coro-mise}]

In choosing the optimal order of bandwidth, one can consider choosing $(h_1,\ldots, h_m)$ such that the  mean integrated squared error is minimized. Note that
\begin{align}
\widehat{F}_E(x)-F(x) =\text{Jacob}(\mathcal{P})_{\mu(x)}\left(\widehat{F}(x)-\mu(x)\right)+o_p(1).
\end{align}
Here $\text{Jacob}(\mathcal P)$ is the Jacobian matrix of the projection map $\mathcal P$. 
One has
\begin{align*}
MISE (\widehat{F}_E(x))&=\int E \|\widehat{F}_E(x)-F(x)   \|^2dx\\
&=\int E \|\text{Jacob}(\mathcal{P})_{\mu(x)}\left(\widehat{F}(x)-\mu(x)\right)+o_p(1)  \|^2dx\\
&=\int E\left(\sum_{i=1}^D\left(\sum_{j=1}^D \mathcal P_{ij}\left(\widehat F_j(x)-\mu_j(x) \right) \right)^2+o_p(1)\right)dx\\
&=O(1/n|H|)+\ldots+O(1/n|H|)+O(h_1^4)+\ldots+O(h_m^4).
\end{align*}
The last terms follow from Fatou's lemma, and that the Jacobian map is differentiable at $\mu(x)$ for every $x$. Therefore, if $h_i$'s ($i=1,\ldots, m$) are taken to be of the same order, that is, of $O(n^{-1/(m+4)})$, then one can obtain  MISE($\widehat{F}_E(x)$) with an order of $O(n^{-4/(m+4)}).$
\end{proof}


\begin{proof}[Proof of Theorem \ref{th-uniform}]
Let  $B$ be the $D\times d$ matrix of the differential $d_{\widetilde{\mu}(x)}\mathcal P$ with respect to given orthonormal basis of tangent space $T_{\widetilde{\mu}(x)}\R^D$ and tangent space $T_{\mathcal P\widetilde{\mu}(x)}\widetilde M$. Given a canonical choice of basis for tangent space $T_{\widetilde{\mu}(x)}\R^D$, one has 
the representation for 
\begin{align}
\sup_x\|d_{\widetilde\mu(x)}\mathcal{P} \left ( \widehat{F}(x)-E(\widehat{F}(x))\right)\|=\sup_x\sqrt{\sum_{i=1}^d \left(\sum_{j=1}^D B^T_{ij}\left (\widehat{F}_j(x)-E(\widehat{F}_j(x))\right)\right)^2}.
\end{align}
Note that the projection map is differentiable around the neighborhood of $\mu(x)$ and $\mathcal X$ is compact, so  $B^T_{ij}(x)$ are bounded. 
Let $C_{ij}=\sup_{x\in \mathcal X} (B^T_{ij})^2(x)$ and $C=\max{C_{ij}}$. 
For each term note that, by Cauchy-Schwarz inequality, 
\begin{align}
\sup_{x\in\mathcal X}\left(\sum_{j=1}^D\left(B^T_{ij}\left(\widehat F_j(x)-E\left(\widehat F_j(x)\right) \right) \right)\right)^2&\leq \sup_x \sum_{j=1}^D(B^T_{ij})^2\left(\widehat F_j(x)-E\left(\widehat F_j(x)\right) \right)^2 \\
&\leq C\sum_{j=1}^D\sup_{x\in \mathcal X}\left(\widehat F_j(x)-E\left(\widehat F_j(x)\right) \right)^2.
\end{align}
By   Theorem 2   in \cite{uniformrates},  one can see that
\begin{align}
\sup_{x\in \mathcal X}|\left(\widehat F_j(x)-E\left(\widehat F_j(x)\right) \right)| =O(r_n),
\end{align}
where $r_n=\log^{1/2} n/\sqrt{n|H|}$.  Then one has
\begin{align}
\sup_{x\in\mathcal X}\sum_{i=1}^d\left(\sum_{j=1}^D\left(B_{ij}\left(\widehat F_j(x)-E\left(\widehat F_j(x)\right) \right) \right)\right)^2=O(r_n^2).
\end{align}
Then one has
\begin{align*}
\sup_x\|d_{\widetilde\mu(x)}\mathcal{P} \left ( \widehat{F}(x)-E(\widehat{F}(x))\right)\|&=\sup_{x\in\mathcal X}\sqrt{\sum_{i=1}^d\left(\sum_{j=1}^D\left(B_{ij}\left(\widehat F_j(x)-E\left(\widehat F_j(x)\right) \right) \right)\right)^2}\\
&=O(r_n)=O\left(\log^{1/2} n/\sqrt{n|H|}\right).
\end{align*}

\end{proof}

\begin{proof}[Proof of Theorem \ref{th-poly}]

Given the higher order smoothness assumption on $\mu(x)$, one can make higher order approximations and using a local polynomials regression estimate would result in the reduction of bias term in estimating $\mu(x)$.   The asymptotic distribution for multivariate local regression estimator for Euclidean responses has been derived \citep{guetal, ruppert1994, masryJTSA571},   and we leverage on some of their results in our proof.

Note that $\widehat{F}(x)=\left( \widehat{F}_1(x),\ldots, \widehat{F}_D(x) \right)\in \R^D$. $E(\widehat{F}(x))=\left(E(\widehat{F}_1(x)),\ldots, E(\widehat{F}_D(x)) \right)^T$ and the expectation taken in each component is with respect to the marginal distribution of $\widetilde{P}(dy|x)$. Then by Theorem 1 of \cite{guetal}, the following holds:
\begin{itemize}
\item[(1)] If $p$ is odd, then for $j=1,\ldots, D$
\begin{align}
\label{eq-polybias1}
\Bias_j(\widehat{F}(x))&\nonumber=E(\widehat{F}_j(x)  )-\mu_j(x)\\
&=\left( \mathcal M_p^{-1}\mathcal B_{p+1}\boldsymbol H^{(p+1)}\boldsymbol m^j_{\boldsymbol{p+1}}(x)\right)_1,
\end{align}
which is of order $O(\|\boldsymbol h\|^{p+1})$.
Here  $(\cdot)_1$ represents the first entry of the vector inside the parenthesis;
\item[(2)] If $p$ is even, then  for $j=1,\ldots, D$
\begin{align}
\label{eq-polybias2}
&\Bias_j(\widehat{F}(x))=E(\widehat{F}_j(x)  )-\mu_j(x)\\ \nonumber
&=\left( \sum_{l=1}^mh_l\dfrac{f_l(x)}{f_X(x)} \left(\mathcal M_p^{-1}\mathcal B^l_{p+1}-\mathcal M_p^{-1}\mathcal M_p^{l}\mathcal M_p^{-1}\mathcal B_{p+1}\right)\boldsymbol H^{(p+1)}\boldsymbol m^j_{\boldsymbol{p+1}}(x)+\mathcal M_p^{-1}\mathcal B_{p+2}\boldsymbol H^{(p+2)} \boldsymbol m^j_{\boldsymbol{p+2}}(x)\right)_1,
\end{align}
which is of order $O(\|\boldsymbol h\|^{p+2})$.
\end{itemize}
For any $k\in\{0,1,\ldots, p\}$. Let $N_k=\binom{k+m-1}{m-1}$  and $\mathcal{N}_{p}=\sum_{k=0}^pN_k$. Here $\mathcal M_p$ is a $\mathcal{N}_{p}\times \mathcal{N}_{p}$ matrix whose  $(i,j)$th block ($0\leq i,j\leq p$) is given by $\int_{\R^m} \boldsymbol u^{i+j}K(\boldsymbol u)d\boldsymbol u$ and $\mathcal M^l_p$ ($l=1,\ldots, m$) is a $\mathcal{N}_{p}\times \mathcal{N}_{p}$ matrix whose  $(i,j)$th block ($0\leq i,j\leq p$) is given by $\int_{\R^m}u_l \boldsymbol u^{i+j}K(\boldsymbol u)d\boldsymbol u$. $\mathcal B_{p+1}$ is a $\mathcal{N}_{p}\times N_{p+1}$ matrix whose $(i, p+1)$th ($i=1,\ldots, p$) block is given by $\int_{\R^m} \boldsymbol u^{i+p+1}K(\boldsymbol u)d\boldsymbol u$ and $\mathcal B_{p+1}^l$ ($l=1,\ldots, m$) is a $\mathcal{N}_{p}\times N_{p+1}$ matrix whose $(i, p+1)$th ($i=1,\ldots, p$) block is given by $\int_{\R^m} u_l\boldsymbol u^{i+p+1}K(\boldsymbol u)d\boldsymbol u$. We have $\boldsymbol H^{(p+1)}=\diag\{h_1^{p+1},\ldots, h_m^{p+1}\}$. $f_l(x)=\dfrac{\partial f_X(x)}{\partial x^{l}}$ and $\boldsymbol m^j_{\boldsymbol{p+1}}(x)$ ($j=1,\ldots, D$) is the vector of all the $p+1$ order partial derivative of $\mu_j(x)$, that is, $m^j_{\boldsymbol{p+1}}(x)=\left(\dfrac{\partial \mu_j^{p+1}(x)}{\partial (x^1)^{p+1}}, \dfrac{\partial \mu_j^{p+1}(x)}{\partial (x^1)^{p}\partial (x^2)},\ldots,  \dfrac{\partial \mu_j^{p+1}(x)}{\partial (x^m)^{p+1}}  \right)$.

With $\Bias_j(\widehat{F}(x))$ ($j=1,\ldots,D$) given above, one has
\begin{align}
\label{eq-polybias}
\Bias(x)=E(\widehat{F}(x)))-\mu(x)=\left(\Bias_1(\widehat{F}(x)),\ldots, \Bias_D(\widehat{F}(x)) \right)^T.
\end{align}

Although higher order polynomial regression results in the reduction in the order of bias with the higher order smoothness assumptions on $\mu(x)$, the order and expression of the covariance remains the same. That is,
\begin{align}
\label{eq-cov}
\Sigma_{jk}\nonumber&=\cov(\widehat{F}_j(x), \widehat{F}_k(x))\\
&=\frac{1}{n| H|}f_X(x)^{-1}\int K(v)^2\sigma(J_j(y_{v}), J_k(y_{v}))dv+o\left(\frac{1}{n|H| }\right),
\end{align}
where $\sigma(J_j(y_{v}), J_k(y_{v})$ is the covariance between $J_j(y_{v})$ and $J_k(y_{v})$.

Applying the central limit theorem, one has
\begin{align}
\label{eq-cltnoprojpoly}
\sqrt{n|H|} \left (\widehat{F}(x)-\mu(x)-\Bias(x)\right) \xrightarrow{L} N(0,\bar \Sigma(x))
\end{align}
where the $j$th ($j=1,\ldots, D$) entry of $\Bias(x)$ is given in \eqref{eq-polybias1} or \eqref{eq-polybias2}  depending on $p$ is odd or even, and
\begin{align}
\label{eq-sigbar2}
\bar \Sigma_{jk}=\dfrac{\sigma(J_j(y_{v}), J_k(y_{v}))\int K(v)^2dv}{f_X(x)}.
\end{align}
Letting $\widetilde{\mu}(x)=\mu(x)+\Bias(x)$, one has
\begin{align*}
\sqrt{n|H|}\left(\widehat{F}_E(x)-\mathcal P \left(\widetilde{\mu}(x)\right)\right)=\sqrt{n|H|}d_{\widetilde{\mu}(x)}\mathcal P\left (\widehat{F}(x)- \widetilde{\mu}(x)\right)+o_P(1).
\end{align*}
Therefore by applying Slutsky's theorem,  one has
\begin{align}
\sqrt{n|H|}d_{\widetilde{\mu}(x)}\mathcal P\left (\widehat{F}(x)- \widetilde{\mu}(x)\right) \xrightarrow{L} N(0,\widetilde \Sigma(x)).
\end{align}
Here $\widetilde\Sigma(x)=B^T\bar \Sigma(x) B$ where $B$ is the $D\times d$ matrix of the differential $d_{\widetilde{\mu}(x)}P$ with respect to given orthonormal bases of the tangent space $T_{\widetilde{\mu}(x)}\R^D$ and tangent space $T_{\widetilde{\mu}(x)}\widetilde{M}$.

\end{proof}

\bibliographystyle{apalike}
\bibliography{reference1-1.bib}

\end{document}